\documentclass{article}

\usepackage[utf8]{inputenc}
\usepackage{amsmath}
\usepackage{amsfonts}
\usepackage{amssymb}
\usepackage{amsthm}
\usepackage{bbm}
\usepackage{biblatex}
\usepackage{enumitem}
\usepackage{geometry}
\usepackage{graphicx}
\usepackage{import}
\usepackage{listings}
\usepackage{multicol}
\usepackage{soul}
\usepackage{standalone}
\usepackage{titlesec}
\usepackage{tikz}

\usepackage[colorinlistoftodos]{todonotes}

\newgeometry{margin=50pt}

\newcommand{\R}{\mathbb{R}}

\newcommand{\C}{\mathcal{C}}

\newcommand{\MAF}{{\rm uMAF}}

\newtheorem{Corollary}{Corollary}
\newtheorem{Definition}{Definition}
\newtheorem{Lemma}{Lemma}

\newtheorem{Theorem}{Theorem}

\usetikzlibrary{decorations.pathreplacing, quotes}
\newcommand{\vertex}{\node[vertex]}
\tikzstyle{vertex}=[draw, shape=circle, minimum size=0.5em, inner sep=1, fill]

\title{Convex characters, algorithms and matchings}
\author{Steven Kelk, Ruben Meuwese, Stephan Wagner}
\date{}

\newcommand{\Addresses}{{
  \bigskip
  \footnotesize

  Steven Kelk, Ruben Meuwese, \textsc{Department of Data Science and Knowledge Engineering (DKE), Maastricht University, P.O. Box 616, 6200 MD Maastricht, The Netherlands}\par\nopagebreak
  \textit{E-mail address}, S. Kelk: \texttt{steven.kelk@maastrichtuniversity.nl}\par\nopagebreak
  \textit{E-mail address}, R. Meuwese: \texttt{r.meuwese@maastrichtuniversity.nl}

  \medskip

  Stephan Wagner, \textsc{Department of Mathematics, Uppsala Universitet, P.O. Box 480, 751 06 Uppsala, Sweden}\par\nopagebreak
  \textit{E-mail address}, S. Wagner: \texttt{stephan.wagner@math.uu.se}

}}

\begin{document}

\maketitle

\begin{abstract}
Phylogenetic trees are used to model evolution: leaves are labelled to represent contemporary species (``taxa'') and interior vertices represent extinct ancestors. Informally, convex characters are measurements on the contemporary species in which the subset of species (both contemporary and extinct) that share a given state, form a connected subtree. In \cite{KelkS17} it was shown how to efficiently count, list and sample certain restricted subfamilies of convex characters, and algorithmic applications were given. We continue this work in a number of directions. First, we show how combining the enumeration of convex characters with existing parameterised algorithms can be used to speed up exponential-time algorithms for the \emph{maximum agreement forest problem} in phylogenetics. 
Second, we re-visit the quantity $g_2(T)$, defined as the number of convex characters on $T$ in which each state appears on at least 2 taxa. We use this to give an algorithm with running time  $O( \phi^{n} \cdot \text{poly}(n) )$, where $\phi \approx 1.6181$ is the golden ratio and $n$ is the number of taxa in the input trees, for computation of \emph{maximum parsimony distance on two state characters}. By further restricting the characters counted by $g_2(T)$ we open an interesting bridge to the literature on enumeration of matchings. By crossing this bridge we improve the running time of the aforementioned parsimony distance algorithm to $O( 1.5895^{n} \cdot \text{poly}(n) )$, and obtain a number of new results in themselves relevant to enumeration of matchings on at-most binary trees.
\end{abstract}

\section{Introduction}

In bioinformatics evolution is often represented by a binary tree $T$ whose leaves are bijectively labelled by a set $X$ of labels (``taxa'') which represent contemporary species. Interior nodes represent hypothetical ancestors of the species in $X$; the tree models the branching process which, via evolutionary phenomena such as mutation and speciation, caused $X$ to evolve. Such a tree is known as a \emph{phylogenetic tree}.

There is an extensive mathematical and algorithmic literature on constructing and comparing phylogenetic trees (see e.g. \cite{SempleS03,BulteauW19}). Many corresponding optimisation problems are NP-hard. One algorithmic approach to tackling NP-hard problems is to develop exact exponential-time algorithms \cite{fomin2010exact}: usually, algorithms that run in time $O(c^n \cdot \text{poly}(n))$ where $c$ is a (small) constant and $n$ is the size of the input. The literature on exponential-time algorithms focusses, logically, on optimising the constant $c$. The interest is not purely theoretical, since especially in the early phases of studying an optimisation problem an exponential-time algorithm might be the best practical method available, and then optimising $c$ can make a crucial difference in practice. Exponential-time algorithms can also be useful for solving the smaller, reduced instances yielded by pre-processing strategies \cite{kernelization2019}.

In this article we give improved exponential-time algorithms for three optimisation problems occurring in phylogenetics. In all three cases we make heavy use of enumeration; in particular, enumeration of so-called \emph{convex characters}. For a phylogenetic tree $T$ on $X$, a convex character is a partition of $X$ such that the minimal spanning trees induced by the blocks of $X$ are disjoint in $T$ (see e.g. Section 4.1 of \cite{SempleS03}). In \cite{KelkS17} it was proven that a phylogenetic tree on $n$ taxa has $\Theta(\phi^{2n})$ convex characters, and
$\Theta(\phi^{n})$ convex characters in which each block of the character contains at least two taxa, where $\phi \approx 1.681$ is the golden ratio and $\phi^2 \approx 2.6181$. These characters can also be efficiently counted and listed, which formed the basis for a $O^{*}(\phi^{2n})$ time algorithm for the \emph{unrooted maximum agreement forest} problem and a $O^{*}(\phi^{n})$ time algorithm for the \emph{maximum parsimony distance} problem, where $O^{*}(.)$ denotes suppression of polynomial factors. The algorithms are based on the insight that optimal solutions are either convex characters themselves, or ``project down'' onto convex characters - see \cite{kelk2021sharp} for related discussions.

In this article we extend the results from \cite{KelkS17} in several directions. We improve the $O^{*}(2.6181^n)$ algorithm for the unrooted maximum agreement forest problem to $O^{*}(2.2973^n)$, and give a $O^{*}(2.0649^n)$ algorithm for the \emph{rooted} variant of this problem.
In both cases the high-level idea is to leverage an existing branching-based algorithm that can answer the question, ``Is the optimum at most $k$?'' in time exponential in $k$ (as opposed to $n$). This so-called \emph{fixed parameter tractable} algorithm - see \cite{CyganBook} for an introduction to such algorithms - is run incrementally up to a carefully-chosen threshold $k'$. If the threshold is exceeded, i.e. the optimum is ``large'', the optimum solution can be found by searching through a comparatively small subset of convex characters. 

Our third contribution concerns the NP-hard problem \emph{maximum parsimony distance on two-state characters} \cite{KelkF17,KelkISW16}. This problem has a trivial $O^{*}(2^n)$ time algorithm, but unlike the more general variant of the problem there is no obvious link with convex characters. Here we establish such a link, using it to obtain an enumeration-based algorithm with running time $O^{*}(\phi^n)$, and then improve this to  $O^{*}(1.5895^n)$, also using enumeration. The strengthening to $O^{*}(1.5895^n)$ is potentially of wider interest, because it operates by enumerating \emph{matchings} (pairwise disjoint sets of edges) in transformed variants of the input trees. There is an extensive literature on the enumeration of matchings in trees\footnote{The number of matchings in a tree is often called the \emph{Hosoya Index.}}; see \cite{wagner2010maxima} for a survey. Due to the combinatorial structure of the maximum parsimony distance problem, we in fact only need to consider certainly carefully constrained subsets of matchings, and it is these that we bound/enumerate using a technique from \cite{Rosenfeld21,Rote19}. Interestingly, this translates to a new result on normal (i.e. unconstrained) matchings; in particular, we show that on (not necessarily phylogenetic) trees with in total $n$ nodes, maximum degree 3 and no degree-2 nodes adjacent to each other, there can be at most $O(1.5895^n)$ matchings. This supplements existing upper bounds of $O(\phi^n)$, valid for arbitrary trees, and $O(1.5538^n)$ for trees of maximum degree 3 where all interior nodes have degree 3 (see \cite[Theorem 1 and Remark 5]{andriantiana2011number}; the precise constant is $\sqrt{1+\sqrt{2}}$).

In Section \ref{Sec Pre} we give preliminaries. Section \ref{Sec uMAF} describes our improvements for the maximum agreement forest problem(s) while Sections \ref{Sec DM2 Theory} and \ref{Sec Matching Theory} describe our improvement for the maximum parsimony distance problem on two state characters; the first of these two sections builds the correspondence with convex characters, and the second section builds the correspondence with matchings, the number of which we then bound. In an extended discussion (Section \ref{sec:discussion}) we provide a number of auxiliary insights and (lower) bounds emerging from Sections \ref{Sec DM2 Theory} and \ref{Sec Matching Theory}.


\section{Preliminaries}\label{Sec Pre}
For general background on mathematical phylogenetics we refer to \cite{SempleS03, DressHKMS12}. 
An {\it unrooted binary phylogenetic $X$-tree} is an undirected tree $T =(V(T),E(T))$ where every internal vertex has degree 3 and whose leaves are bijectively labelled by a set $X$, where $X$ is
often called the set of \emph{taxa} (representing the contemporary species). We use $n$ to denote $|X|$ and often simply write \emph{tree} when this is clear from the context.
Two phylogenetic trees $T, T'$ on $X$ are considered equal if there is an isomorphism between them that is identity on $X$.

A {\it character} $f$ on $X$ is a surjective function $f: X\rightarrow \C$ for some set $\C$ of {\em states} (where a state represents some characteristic of the species e.g. number of legs). We say that $f$ is an $r$-state character if $|\C|=r$.
Each character naturally induces a partition of $X$ and here we regard two characters as being equivalent if they both induce the same partition of $X$. Hence, the states can simply be regarded as the blocks of a partition. An {\it extension} of a character $f$ to $V(T)$ is a function $h: V(T)\rightarrow \C$ such that $h(x) = f(x)$ for all $x$ in $X$. For such an extension $h$ of $f$, we denote by $l_{h}(T)$ the number of edges $e=\{u,v\}$ such that $h(u) \neq h(v)$. The {\em parsimony score} of a character $f$ on $T$, denoted by $l_{f}(T)$, is obtained by minimising $l_{h}(T)$ over all possible extensions $h$ of $f$. We say that a character $f: X \rightarrow \C$ is \emph{convex on $T$} if $l_{f}(T) = |\C| - 1$. Equivalently: a character $f: X \rightarrow \C$ is convex on $T$ if there exists an extension $h$ of $f$ such that, for each state $c \in \C$, the vertices of $T$ that are allocated state $c$ (by $h$) form a connected subtree of $T$. We call such an extension $h$ a \emph{convex extension} of $f$. The convexity of a character can be tested in polynomial \cite{Fitch71, Hartigan73} (in fact, linear \cite{BachooreB06}) time. We note that a third, equivalent definition of convexity which does not use the machinery of parsimony scores or extensions is as follows: a character $f$ is convex on $T$ if the minimal spanning trees
induced by the blocks of $f$ are vertex-disjoint.
 In \cite{KelkS17} the quantity $g_k(T)$ was introduced, $k \geq 1$, defined as the number of convex characters of $T$ in which each state contains at least $k$ taxa. The quantity $g_1(T)$ is therefore simply the total number of convex characters. This grows as $\Theta(\phi^{2n})$, independently of $T$, where $\phi \approx 1.6181...$ is the golden ratio. Because the bases of the exponential terms we describe in this article are often close together, and the use of these terms in bounding the worst-case running time of algorithms, we give 4 decimal points and always round up. 
 This will provide a better view of our improvements. In Section \ref{Sec DM2 Theory} of this article, $g_2(T)$ plays a prominent role. It grows as $\Theta(\phi^{n})$, also independently of $T$ \cite{KelkS17}. See Figure \ref{fig:g2} for an illustration.

\begin{figure}[h]
\centering
\includegraphics[scale=0.2]{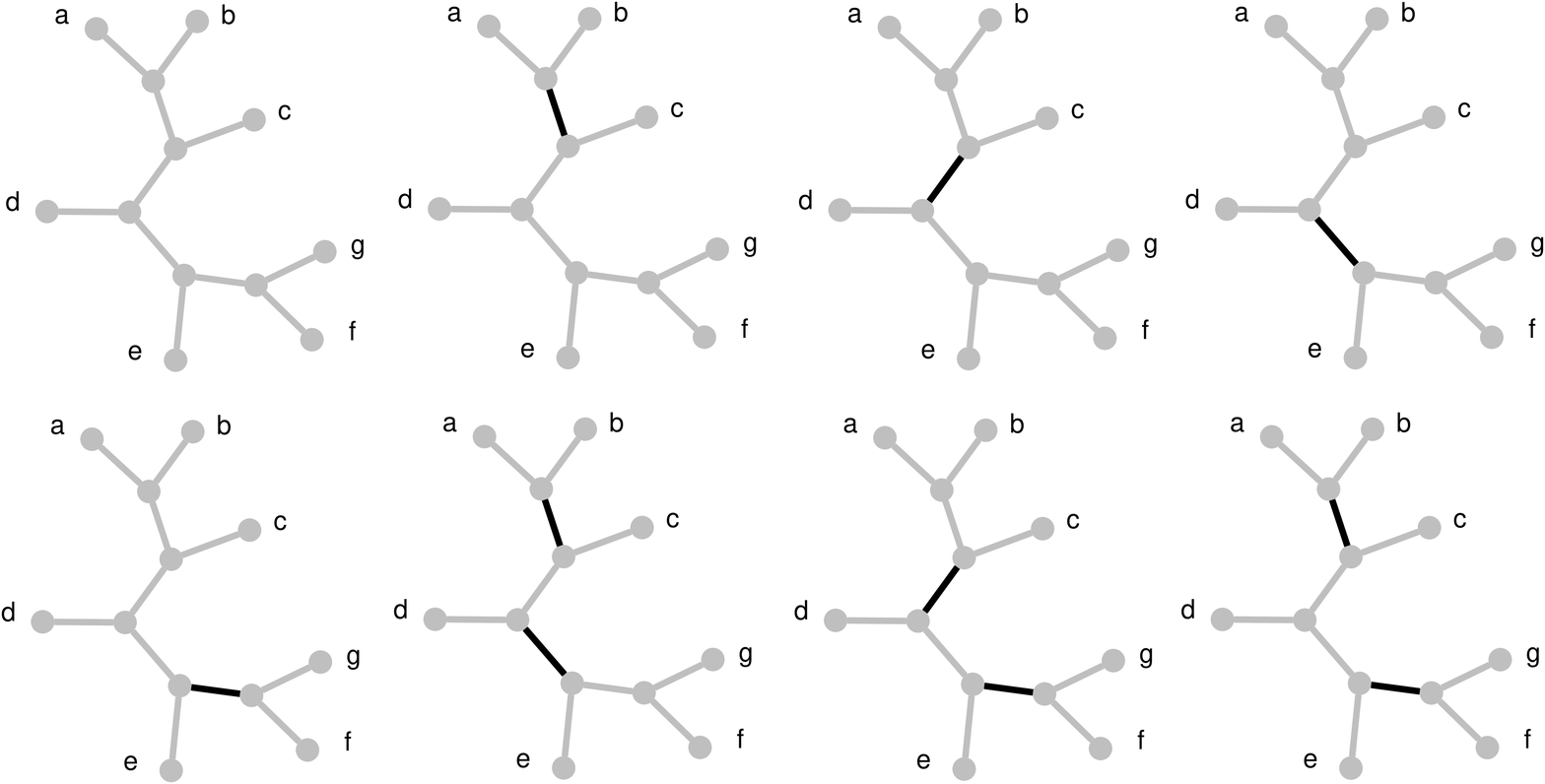}
\caption{A phylogenetic tree $T$ where $X=\{a,b,c,d,e,f,g\}$. For the given tree there are in total 233 convex characters \cite{KelkS17}. There are 8 convex characters in which each state appears on at least 2 taxa, so $g_2(T)=8$. One of these characters has exactly 1 state ($abcdefg$), 4 of these characters have 2 states ($ab|cdefg$, $abc|defg$, $abcd|efg$ and $abcde|fg$), and 3 of these characters have 3 states ($ab|cd|efg$, $abc|de|fg$ and $ab|cde|fg$). Here we use $|$ to denote the states (i.e. blocks) of the character (i.e. partition). The 8 figures verify the convexity of these characters: the minimal spanning trees induced by the states of the character, shown in grey, are vertex-disjoint.}
\label{fig:g2} 
\end{figure}


\section{Improved exponential-time algorithms for agreement forests} \label{Sec uMAF}

Let $T$ be a tree on $X$. For $X' \subseteq X$ we write $T[X']$ to denote the miminal subtree of $T$ spanning $X'$, and
write $T|X$ to denote the phylogenetic tree obtained from $T[X']$ by suppressing nodes of degree 2.

Let $T$ and $T'$ be two phylogenetic trees  on $X$. Let $F = \{ X_1,X_2,\ldots,X_k\}$ be a partition of $X$, where each block $X_i$ with $i\in\{1,2,\ldots,k\}$ is  referred to as a \emph{component} of $F$. We say that $F$ is an \emph{agreement forest} for $T$ and $T'$ \cite{AllenS01} if the following conditions hold.
\begin{enumerate}
\item [(1)] For each $i\in\{1,2,\ldots,k\}$, we have $T|X_i = T'|X_i$.
\item [(2)] For each pair $i,j\in\{1,2,\ldots,k\}$ with $i \neq j$, we have that
$T[X_i]$ and $T[X_j]$ are vertex-disjoint in $T$, and $T'[X_i]$ and $T'[X_j]$ are vertex-disjoint in $T'$.
\end{enumerate}
\noindent
Let $F=\{X_1,X_2,\ldots,X_k\}$ be an agreement forest for $T$ and $T'$. The \emph{size} of $F$ is simply its number of components, $k$. Moreover, an agreement forest with the minimum number of components (over all agreement forests for $T$ and $T'$) is called a \emph{maximum agreement forest (MAF)} for $T$ and $T'$. The number of components of a maximum agreement forest for $T$ and $T'$ is denoted by $d_\MAF(T,T')$. The \textsc{Unrooted Maximum Agreement Forest (uMAF)} problem is to compute $d_\MAF(T,T')$; it is NP-hard \cite{AllenS01,hein1996complexity}. There is also a \emph{rooted} version of the problem but we defer a discussion of this until later.

\subsection{An $O^{*}(2.2973^n)$ algorithm for finding an uMAF}

By part (2) of the definition of agreement forest, an agreement forest for $T$ and $T'$ is a convex character on both $T$ and $T'$, but the converse does not necessarily hold. A tree has $\Theta(\phi^{2n})$ convex characters and these can be listed efficiently \cite{KelkS17}. Combined with the fact that it can easily be tested in polynomial time whether a convex character is an agreement forest, \cite{KelkS17} gave an elementary $\Theta^{*}(\phi^{2n}) = \Theta^{*}(2.6181^n)$ algorithm for computing $d_\MAF(T,T')$, simply by looping through all characters that are convex on $T$ and noting the smallest agreement forest found.

To improve upon this, we leverage an  existing algorithm by Chen et al \cite{ChenFS15} that can answer the question ``Is $d_\MAF(T,T') \leq k$?'' in time $O( 3^{k} \cdot \text{poly}(n))$ -- and, when the answer is YES, can construct an agreement forest of size at most $k$. Starting at 1 and incrementing $k$ until the answer is YES gives an algorithm for computing $d_\MAF(T,T')$. 
Such \emph{fixed parameter tractable} (FPT) algorithms are fast when the parameter $k$ is small (relative to $n$), but slow down considerably for higher $k$. Our idea is to use the algorithm of Chen up to a certain tipping point $k'=cn$, for a constant $c$ still to be determined, exploiting the fact that $3^k$ is far smaller than $2.6181^n$ up to quite large values of $k$. If Chen's algorithm finds an agreement forest, we are done. Otherwise (i.e.  the answer is still NO) we switch to enumeration of convex characters. Specifically, we list all characters that are convex on $T$ and which have more than $k'$ states. For each such character we check whether it is an agreement forest of $T$ and $T'$ and note the smallest agreement forest found this way.


To put this into practice, we need to bound the number of convex characters with `many' states, ensure that they can be listed/enumerated efficiently, and determine the optimal tipping point $k'$.\\
\\
Let us define the \emph{size} of a convex character as its number of states/blocks. We will show that there is a dynamic programming (DP) algorithm with a running time of $O(f(k) \cdot \text{poly}(n))$ that can enumerate every convex character with $k$ or more states for a tree $T$ where $f(k)$ equals the number of convex characters of size at least $k$.  Steel \cite{Steel92} gives us an exact expression for how many convex character there are of size $k$, denoted by $g(T, k)$ which equals $\binom{2n-k-1}{k-1}$ - note that this number is independent of the topology of $T$. The whole point of designing this DP is not to calculate how many convex characters there are but where they come from, so they can be efficiently constructed. Kelk and Stamoulis \cite{KelkS17} have given a DP algorithm that enumerates convex characters where each state has at least a certain number of taxa. We draw inspiration from this algorithm and construct a similar DP algorithm. The DP requires the tree to be rooted, in order to establish an unambiguous parent-child relation between nodes. This can be easily achieved by subdividing an edge of the tree, since the presence of the subdividing vertex does not change the space of convex characters.\\
\\
For a node $u$ of a tree $T$, we let $T_u$ be the subtree rooted at $u$. A convex character of size $k$ for $T_u$ can be created in two ways. It can be the disjoint union of a convex character from the left child node of $u$, and the right, such that the combination has $k$ states. Or it can be a combination of a convex character from the left child node and the right whereby a state from the left character is merged with a state from the right i.e. the spanning tree for this merged state traverses the root.
Because of this we need to store more than simply all the convex characters of subtrees in our DP algorithm; we also have to note which states of a convex character can `reach the root'; we will formalise this in due course.

To keep track of these distinctions, for each node $u$ of tree $T$ we enumerate ordered pairs $(f,A)$ where $f$ is a convex character of $T_u$ such that there is an optimal extension of $f$ to $T_u$ whereby the root of $T_u$ is assigned state $A$. Essentially,
this means that, if desired, state $A$ of $f$
will be available for merging with another state in the tree beyond $T_u$.


We write $(f, \emptyset)$ to denote
the situation when $f$ is a convex character of $T_u$ but we simply do not care about which state is assigned to the root of $T_u$. We will not be merging any state of this character with a state from beyond $T_u$, which is why we do not care.


For the DP algorithm we store two types of values. First, the number of convex characters of size $k$, denoted by $g(T_u, k)$, and second the number of pairs $(f,A)$, $A \neq \emptyset$, such that $f$ has $k$ states, denoted by $h(T_u, k)$. Note that $g(T_u,k)$ simply counts all $(f,\emptyset)$ for $T_u$ where $f$ has size $k$.

Let us consider a simple example. Suppose $u$ has two children, both
taxa, labelled $a$ and $b$. $T_u$ has two convex characters, $\{\{a,b\}\}$ and $\{\{a\},\{b\}\}$. 
In this case $g(T_u,1)=1$ and $g(T_u,2)=1$. We have $h(T_u,1)=1$ because this counts the unique tuple $(f, \{a,b\})$ where $f$ is
$\{\{a,b\}\}$. We also have $h(T_u,2)=2$ because
this counts the tuples $(f, \{a\})$ and $(f, \{b\})$ where $f$ is $\{\{a\},\{b\}\}$.

Let $L(T_u)$ be the set of leaves of the subtree $T_u$ and let $m_u=|L(T_u)|$. This means that in the equations below $k$ (and thus $i$ and $j$) have a maximum value $m_u$.

We compute the values $g(T_u, k)$ and $h(T_u, k)$ for each value $k$ using two recursions and move bottom-up from the leaves. If we assume for a node $u$ that the left node $l$ and right node $r$ are calculated correctly the values of $u$ are calculated as follows:
\begin{equation} \label{h value equation FOR REAL}
    h(T_u, k) = \sum_{\substack{i, j \\ i+j=k}} h(T_l, i) g(T_r, j) + g(T_l, i) h(T_r, j)\ \ +\sum_{\substack{i, j \\ i+j=k+1}} h(T_l, i) h(T_r, j)
\end{equation}

Each pair $(f,C)$ in $T_u$ that is counted by the $h(T_u, .)$ quantities has 3 possible origins, demonstrated in Figure \ref{Fig h(T,k) origins}.
\begin{enumerate}
    \item A pair $(f_l,A)$ from the left child node, $A \neq \emptyset$,  is combined with a pair $(f_r, \emptyset)$ from the right child node to create the pair $(f_l \cup f_r, A)$. The character $f_l \cup f_r$ has size $|f_l| + |f_r|$.
    \item 
    A pair $(f_l,\emptyset)$ from the left child node is combined with a pair $(f_r, B)$ from the right child node, $B \neq \emptyset$,  to create the pair $(f_l \cup f_r, B)$. The character $f_l \cup f_r$ has size $|f_l| + |f_r|$.
    \item A pair $(f_l,A)$ from the left child, $A \neq \emptyset$, is combined with a pair $(f_r, B)$ from the right child node, $B \neq \emptyset$, to create the pair $(f_l \cup f_r, A \cup B )$. The character $f_l \cup f_r$ has size $|f_l| + |f_r|- 1$.
\end{enumerate}

The first two origins are counted in the first sum of Equation \ref{h value equation FOR REAL}. The third origin is counted in the second sum of that equation.


\begin{figure}[!h]
\begin{center}
\begin{tikzpicture}[scale = 0.8]
	\vertex [label=left:$u$] (u) at (0,0) {};
	\vertex [label=left:$l$] (l) at (-1.5,-1) {};
    \vertex [label=right:$r$] (r) at (1.5,-1) {};
    
    \node (Tl) at (-1.5,-1.75) {$(f_l,A)$};
    \node (Tr) at (1.5,-1.75) {$(f_r, \emptyset)$};
	
	\draw [line width = 1pt]
	(u) edge [red] (l)
	(u) edge (r)
    (u) edge [red] (0,1)
    (l) -- (-2.5,-2) -- (-0.5,-2) -- (l)
    (r) -- (0.5,-2) -- (2.5,-2) -- (r);
    
    \vertex [label=left:$u$] (u) at (6,0) {};
	\vertex [label=left:$l$] (l) at (4.5,-1) {};
    \vertex [label=right:$r$] (r) at (7.5,-1) {};
    
    \node (Tl) at (4.5,-1.75) {$(f_l, \emptyset )$};
    \node (Tr) at (7.5,-1.75) {$(f_r, B)$};
	
	\draw [line width = 1pt]
	(u) edge (l)
	(u) edge [red] (r)
    (u) edge [red] (6,1)
    (l) -- (3.5,-2) -- (5.5,-2) -- (l)
    (r) -- (6.5,-2) -- (8.5,-2) -- (r);
    
    \vertex [label=left:$u$] (u) at (12,0) {};
	\vertex [label=left:$l$] (l) at (10.5,-1) {};
    \vertex [label=right:$r$] (r) at (13.5,-1) {};
    
    \node (Tl) at (10.5,-1.75) {$(f_l, A)$};
    \node (Tr) at (13.5,-1.75) {$(f_r, B)$};
	
	\draw [line width = 1pt]
	(u) edge [red] (l)
	(u) edge [red] (r)
    (u) edge [red] (12,1)
    (l) -- (9.5,-2) -- (11.5,-2) -- (l)
    (r) -- (12.5,-2) -- (14.5,-2) -- (r);
	
\end{tikzpicture}
\end{center}
\caption{All 3 possible origins for a pair $(f,C)$, as counted by $h(T_u,.)$. }
\label{Fig h(T,k) origins}
\end{figure}
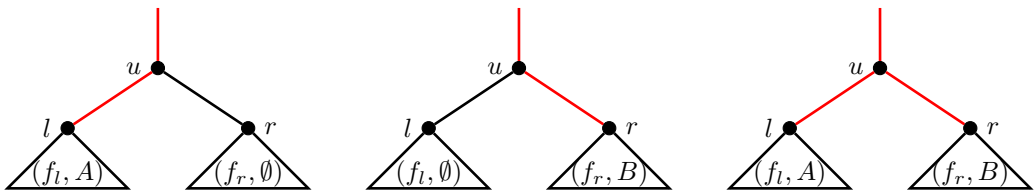

The equation for $g(T_u, k)$ is as follows. The logic behind this equation is similar to that behind $h(T_u, k)$; see Figure \ref{Fig g(T,k) origins}. 

\begin{equation} \label{g value unsorted}
    g(T_u, k) = \sum_{\substack{i, j \\ i+j=k}}  g(T_l, i) g(T_r, j) + \sum_{\substack{i, j \\ i+j=k+1}}  h(T_l, i) h(T_r, j)
\end{equation}

The only thing that remains is a way to enumerate these convex characters. In other words, a way to backtrack in the tree $T$. The recursion will be this guide. One way is looking at Equation \ref{h value equation FOR REAL} as if it were a list that starts with $h(T_l, 1) g(T_r, k-1)$, followed by $h(T_l, 2) g(T_r, k-2)$, passing through $g(T_l, 1) h(T_r, k-1)$ and ending at $h(T_l, k) h(T_r, 1)$. Rewriting this will show it in a more formal way.

\begin{equation} \label{h value sorted}
    h(T_u, k) = \sum_i h(T_l, i) g(T_r, k-i) + \sum_i g(T_l, i) h(T_r, k-i) +\sum_i h(T_l, i) h(T_r, k+1-i)
\end{equation}

We do the same thing for $g(T_u, k)$.

\begin{equation} \label{g value sorted}
    g(T_u, k) = \sum_i g(T_l, i) g(T_r, k-i) + \sum_i h(T_l, i) h(T_r, k+1-i)
\end{equation}

\begin{figure}[!h]
\begin{center}
\begin{tikzpicture}[scale = 0.8]
	\vertex [label=left:$u$] (u) at (0,0) {};
	\vertex [label=left:$l$] (l) at (-1.5,-1) {};
    \vertex [label=right:$r$] (r) at (1.5,-1) {};
    
    \node (Tl) at (-1.5,-1.75) {$(f_l,\emptyset)$};
    \node (Tr) at (1.5,-1.75) {$(f_r, \emptyset)$};
	
	\draw [line width = 1pt]
	(u) edge (l)
	(u) edge (r)
    (u) edge (0,1)
    (l) -- (-2.5,-2) -- (-0.5,-2) -- (l)
    (r) -- (0.5,-2) -- (2.5,-2) -- (r);
    
    \vertex [label=left:$u$] (u) at (6,0) {};
	\vertex [label=left:$l$] (l) at (4.5,-1) {};
    \vertex [label=right:$r$] (r) at (7.5,-1) {};
    
    \node (Tl) at (4.5,-1.75) {$(f_l, A )$};
    \node (Tr) at (7.5,-1.75) {$(f_r, B)$};
	
	\draw [line width = 1pt]
	(u) edge [red] (l)
	(u) edge [red] (r)
    (u) edge [red] (6,1)
    (l) -- (3.5,-2) -- (5.5,-2) -- (l)
    (r) -- (6.5,-2) -- (8.5,-2) -- (r);
	
\end{tikzpicture}
\end{center}
\caption{The two possible origins for a convex character with $k$ states.}
\label{Fig g(T,k) origins}
\end{figure}
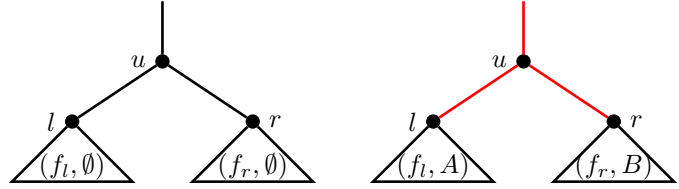

Thus finally we have a way to enumerate\footnote{We note that, if desired, the term $g(T_u,k)$ can be replaced by Steel's closed expression $\binom{2n-k-1}{k-1}$. However, when enumerating and constructing such characters the explicit recurrence for
$g(T_u,k)$ is required.}
 every convex character of size $k$.

\begin{Theorem}
    For a pair of binary trees $T$ and $T'$ we can compute and enumerate each convex character of size at least $k$ for $T$, and determine whether it is a uMAF, in $O(f(k) \cdot \text{poly}(n))$ time.
\end{Theorem}

\begin{proof}
    The DP algorithm described above can easily be leveraged to impose a canonical ordering on the convex characters being counted. This is because the $h$ and $g$ values are cleanly defined as a summation of terms. 
    For each node we calculate at most $n$ different $h(T_u, k)$ and $n$ different $g(T_u,k)$ values, each taking at most linear time to calculate. Then, for each convex character we backtrack through the dynamic programming tree in linear time. We then test if the character is also convex on $T'$ which can be done in linear time according to Bachoore and Bodlaender \cite{BachooreB06}. If the character is convex on both trees, testing if it is an agreement forest can also be done in
    low-order polynomial time.
    
\end{proof}

To get the overall runtime $O^{*}(2.2973^n)$ we need an expression for $f(k)$ and to solve the equation $3^k = f(k)$. This equation arises because we wish to balance the time taken by Chen's FPT algorithm, which is $O^{*}(3^k)$, with the time required to enumerate convex characters after Chen's algorithm has stopped.

An expression for $f(k)$ is obtained easily enough by using Steel \cite{Steel92} once more.

\[
    f(k) = \sum_{r=k}^n \binom{2n - r -1}{r-1}
\]

Let us analyse this sum: the quotient of two consecutive terms is
\[
\frac{\binom{2n-r-2}{r}}{\binom{2n-r-1}{r-1}} = \frac{(2n-2r)(2n-2r-1)}{r(2n-r-1)},
\]
which is greater than $1$ for $r < r_0 = \frac{10n-3 - \sqrt{20n^2-20n+9}}{10} = (1-\frac{1}{\sqrt{5}})n + O(1)$, and less than $1$ for $r > r_0$. In other words, the terms increase up to $r_0$ (provided that $k < r_0$), and decrease afterwards. We can combine this with the bound $\binom{n}{pn} \leq e^{n H(p)}$, where $H(p)$ is the entropy function $H(p) = -p\ln p - (1-p)\ln(1-p)$, see for instance Example 11.1.3 in \cite{CoverT06}. Writing $k = cn$, we have the following bound on $f(k)$:
\[
f(k) \leq \begin{cases} \exp \Big( (2-c) n H\big( \frac{c}{2-c} \big) + O(\log n) \Big) & c \geq c_0 = 1 - \frac{1}{\sqrt{5}}, \\
\exp \Big( (2-c_0) n H\big( \frac{c_0}{2-c_0} \big) + O(\log n) \Big) & c \leq c_0 = 1 - \frac{1}{\sqrt{5}}. \end{cases}
\]

%
%
%
%
%

We compare this with the running time of the FPT algorithm, i.e. $3^k$. Since subexponential factors do not matter to us, this amounts to the equation
\[
    e^{(2-c)nH(\frac{c}{2-c})} = 3^{cn},
\]
as one easily finds that $e^{(2-c_0)H(\frac{c_0}{2-c_0})} > 3^c$ for $c < c_0$. Solving the equation above for $c$ gives a numerical value of $c \approx 0.7571$, and finally the desired runtime of $O^{*}(2.2973^n)$. Concluding:

\begin{Theorem}
Given two trees $T, T'$ on $n$ taxa, an unrooted maximum agreement
forest for $T$ and $T'$ can be computed in time $O^{*}(2.2973^n)$.
\end{Theorem}

\subsection{An $O^{*}(2.0649^n)$ algorithm for finding a rMAF}

A \emph{rooted} agreement forest is a variant of an agreement forest defined on rooted trees i.e. trees where there is a designated root and all arcs are oriented away from it. The corresponding optimisation problem \textsc{Rooted Maximum Agremeent Forest} (rMAF) is also NP-hard. We omit the definition and refer to recent survey articles such as \cite{BulteauW19} for an overview.

There is a $O(2.42^k \cdot \text{poly}(n))$ time algorithm for 
determining whether two rooted trees $T, T'$ have a rooted agreement forest of size at most $k$  \cite{WhiddenBZ13}. In the worst case iterative deepening will yield an algorithm with running time $O^{*}(2.42^n)$. Prior to this article this was the best-known exponential-time (as opposed to parameterised) algorithm for this problem. With the same arguments as in the previous section  we can reduce the base of the exponent. It is easy to show that, given a convex character on $T$, one can test in polynomial time if this induces a rooted agreement forest for $T$ and $T'$. Thus we only have to solve the following equation for $c$.
\[
    2.42^{cn} = e^{(2-c)nH(\frac{c}{2-c})}
\]
This yields $c \approx 0.8204$,
which in turn results in a runtime of $O^{*}(2.0649^n)$.

\begin{Theorem}
Given two rooted trees $T, T'$ on $n$ taxa, a rooted maximum agreement
forest for $T$ and $T'$ can be computed in time $O^{*}(2.0649^n)$.
\end{Theorem}


\section{Maximum parsimony distance on two-state characters}\label{Sec DM2 Theory}

Recall from the preliminaries the definition of the \emph{parsimony score} $l_f(T)$ where $f$ is a character and $T$ is a phylogenetic tree on $X$. This is the foundation of the following optimisation
problem, introduced in \cite{fischer2014,moulton2015}. Informally, the goal is to identify a character that has high parsimony score on one of the input trees and small on the other.

\begin{Definition}
Given two phylogenetic trees $T$ and $T'$ on the same set $X$ of taxa, the parsimony distance between these trees is defined as
\[
    d_{MP}(T, T') = \max_f |l_f(T) - l_f(T')|
\]

A character f which maximises this distance is called an optimal character. 
\end{Definition}

Note that, due to the fact that the parsimony score of a tree (for a given character) is not affected by the presence or absence of a root, parsimony distance is oblivious to whether the input trees are rooted or unrooted.

The problem is NP-hard and has attracted quite some attention in recent years. See recent articles such as \cite{jones2021maximum,van2020reflections} for an overview of its theoretical and practical significance; in particular, it can be leveraged to efficiently generate very strong lower bounds on $d_\MAF(T,T')$. Kelk and Stamoulis \cite{KelkS17} described and implemented an $O^{*}(1.6181^n)$ algorithm that calculates $d_{MP}$ for two trees. The algorithm exploits the fact, proven in \cite{KelkF17}, that there exists an optimal character $f$ which is convex on one of the two input trees and which has at least two taxa per state.

In this section we study a different variant of the problem,
known as $d_{MP}^2$. The definition is the same except that
we are restricted to characters $f$ that only have two states. This problem, which yields lower bounds on $d_{MP}$, is also NP-hard \cite{KelkF17} and permits a trivial
$O^{*}(2^n)$ time algorithm simply by guessing the state of each taxon in $X$; see also \cite{KelkISW16} for further discussion of the problem. Interestingly, the $O^{*}(1.6181^n)$ time algorithm for the $d_{MP}$ problem does not work for $d^2_{MP}$. This is because the pivotal assumption, that there exists an optimal
character that is convex on one of the two trees, no longer holds when restricted to just two states. Nevertheless, with some effort we can establish a link with convex characters. We will use this link to prove the same $O^{*}(1.6181^n)$ runtime for $d_{MP}^2$ - and then via matchings we will reduce the base of the exponent further.\\
\\
We say that an extension $g$ of a character $f$ on a tree $T$ 
    is \emph{optimal} if $l_{g}(T) = l_{}(T)$ i.e. it is an extension which achieves the minimum number of mutations. From this point on we will refer to the states in a two-state character as \emph{red} and \emph{blue}. 


Before proving our main theorem we
first require some auxiliary lemmas.

\begin{Lemma} \label{Lemma internal node neighbor color}
    Let $T$ be a tree on $X$ and $f_2$ a two-state character on $X$. Let $g_2$ be an optimal extension of $f_2$ to $T$. For every internal node $u$ of $T$, at most one neighbour of $u$ has a colour different to $g_2(u)$. 
\end{Lemma}

\begin{proof}
    Suppose that there is an internal node $u$ which does have more than one neighbour coloured with a colour different to $g_2(u)$. This means that it has two or three neighbours with a different colour. Suppose it has three. Let w.l.o.g. $u$ be red and all its neighbours blue. Flipping $u$ to blue would reduce the number of mutations by three, contradicting the fact that $g_2$ is an optimal extension. Suppose it has two. In this case flipping the colour of $u$ reduces the number of mutations by one which again contradicts the fact that $g_2$ is optimal.
\end{proof}

We obtain a small but useful corollary 
thanks to this lemma. 

\begin{Corollary} \label{Corollary leafs are taxon}
Let $T$ be a tree on $X$ and $f_2$ a two-state character on $X$. Let $g_2$ be an optimal extension of $f_2$ to $T$. Consider the forest of connected components obtained by deleting the mutation edges in $g_2$, i.e. edges $\{u,v\}$ where $g_2(u) \neq g_2(v)$. Every vertex with degree 0 in this forest is labelled by a taxon from $X$, and every leaf (i.e. degree 1 node) in this forest is labelled by a taxon from $X$.
\end{Corollary}
\begin{proof}
If there is a degree-0 node $u$ in the forest that is not labelled by a taxon from $X$, then $u$ is an internal node of $T$, and the colour assigned to $u$ by $g_2$ is distinct from that assigned to all 3 of its neighbours; this contradicts Lemma \ref{Lemma internal node neighbor color}. Similarly, suppose there is a leaf $u$ in a connected component of the forest that is not labelled by a taxon. The leaf must, again, be an internal node of $T$. Meaning that, in $g_2$, it has two neighbours with a different colour to $g(u)$. This also contradicts Lemma \ref{Lemma internal node neighbor color}.
\end{proof}



The following lemma resembles  Observation 3.3 from Kelk and Fischer \cite{KelkF17} but is now rewritten in our notation and with an extended proof.

\begin{Lemma} \label{Lemma 2 taxon ccc}
Let $T$ and $T'$ be the input to the problem $d^2_{MP}$. There
exists an optimal two-state character $f_2$, a tree
$T^{*} \in \{ T, T'\}$ and an optimal extension $g_2$ of $f_2$
to $T^{*}$, such that every connected component of the
forest obtained by cutting mutation edges of $T^{*}$ (with
respect to $g_2$), contains at least two taxa.
\end{Lemma}

\begin{proof}
Let $f_2$ be an optimal character. Assume without loss of generality that $l_{f_2}(T) \leq l_{f_2}(T')$.
Let $g_2$ be an optimal extension of $f_2$ to $T$ and consider the forest induced by cutting the mutation edges of $g_2$ in $T$.
Suppose that there is a connected component in the forest with a single node $x$. From Corollary \ref{Corollary leafs are taxon} we know that $x$ is labelled by a taxon, and therefore must have had a different colour to its unique parent in $g_2$. Flipping the colour of $x$ will create a new two-state character $f'_2$. Let $g'_2$ be the extension of $f'_2$
to $T$ which assigns states to interior nodes
of $T$ in the same way as $g_2$. Extension $g'_2$ has one fewer mutation than $g_2$, so 
$l_{f'_2}(T) \leq l_{f_2}(T) - 1$. However,
$l_{f'_2}(T') \geq l_{f_2}(T') - 1$. To see why
this is, observe that flipping the state of a single taxon can reduce the parsimony score of a character on a given tree by at most one. Hence,
$l_{f'_2}(T') - l_{f'_2}(T) \geq l_{f_2}(T') - l_{f_2}(T) = d^2_{MP}(T,T')$, so $f'_2$ is also an optimal character, and $g'_2$ is an optimal extension of it to $T$. Observe that the forest induced by $g'_2$ has one fewer singleton component than $g_2$. We iterate this process until we have an optimal two-state character $f^{*}_2$, and a corresponding optimal extension $g^{*}_2$ of it to $T$, such that no singleton components appear in the induced forest. At this point, every connected component in the induced forest contains at least two nodes, and thus at least two leaves; by Corollary \ref{Corollary leafs are taxon} these leaves must be labelled by taxa. So we are done.
\end{proof}

We define a \emph{partial} extension of a character to a tree $T$ similarly to an extension, except that some interior nodes of $T$ potentially have no state assigned to them; we call these nodes \emph{uncovered}.

\begin{Definition}
    Consider a convex character f on a tree $T$. The natural partial extension is the partial extension obtained by, for each state, taking the minimal connecting subtree of all the taxa within that state. If this partial extension covers every node of $T$ (i.e. it is an extension) we call it the natural covering extension.
\end{Definition}

\begin{Lemma} \label{Lemma optimal contain natural}
    Let $f$ be a convex character on a tree $T$. Every optimal extension of $f$ to $T$ can be obtained by expanding the natural partial extension of $f$ to $T$ to the uncovered nodes. In other words, the natural partial extension is unavoidable in an optimal extension.
\end{Lemma}

\begin{proof}
Suppose there is an optimal extension $g$ of $f$ to $T$ that is not an expansion of the natural partial extension. Then, for some state $s$ of $f$, at least one node on its minimal connecting spanning tree is not assigned to state $s$. As a result, at least one edge on the minimal connecting spanning tree for $s$ is a mutation edge in $g$. Hence, if one cuts mutation edges in $g$, the state $s$ will occur in two or more connected components of the resulting forest. Every state of $f$ must occur in at least one connected component of the forest, so the number of connected components in the forest is at least equal to the number of states in $f$ plus 1. The number of connected components in this forest is $l_f(T)+1$, so $l_f(T)$ is at least equal to the number of states in $f$. However, a convex character on $T$ has by definition the property that its parsimony score on $T$ is equal to the number of states in the character \emph{minus one}. Contradiction. 

\end{proof}

The following lemma is actually somewhat stronger than we need - we only require one direction of the implication - but for completeness we prove both directions.

\begin{Lemma} 
\label{strongUnique}
    Let $f$ be a convex character on a tree $T$. The natural covering extension of $f$ exists if and only if there is a \emph{unique} optimal extension of $f$ to $T$.
\end{Lemma}
\begin{proof}
The fact that the existence of the natural covering extension implies uniqueness, follows immediately from Lemma \ref{Lemma optimal contain natural}. To prove the other direction, we use the contrapositive. Assume that the natural covering extension does \emph{not} exist. Then there exists an interior node $u$ that is uncovered by the natural partial extension. Let $g$ be an optimal
extension of $f$ to $T$. We will use an algorithmic argument to prove that $g$ cannot be unique. Let $p, q,r$ be the three neighbours of $u$ in $T$ and let $T_w$, where $w \in \{p,q,r\}$, be the pendant subtree of $T$ rooted at $w$. Let $S_w$, where $w \in \{p,q,r\}$, be the set of states appearing on taxa in $T_w$, with respect to $f$. Crucially, $S_p$, $S_q$ and $S_r$ are mutually disjoint, because otherwise $u$ would have been covered in the natural partial extension. Now, we root $T$ by subdividing the edge $\{u,r\}$. We run Fitch's algorithm on $T$ to compute an optimal extension of $f$ to $T$; we provide a summary of the algorithm in Appendix A. In the bottom-up phase of Fitch's algorithm, a nonempty set $S'_w \subseteq S_w$ of states will be assigned to node $w$, $w \in \{p,q,r\}$. By the disjointness of the $S_w$, the $S'_w$ are also disjoint. Hence, in the bottom-up phase the set $S'_p \cup S'_q \cup S'_r$ of states will be assigned to the root. Now, in the top-down phase, we are allowed to assign any state from $S'_p \cup S'_q \cup S'_r$ to the root. If we pick a state $s \in S'_p$, the node $u$ will definitely be assigned state $s$. If we pick a state $t \in S'_q$, $u$ will definitely be assigned $t$. Given that $s \neq t$, due to disjointness of $S'_q$ and $S'_p$, and the fact that all possible executions of the top-down phase of Fitch's algorithm produce an optimal extension, we conclude that there exists more than one optimal extension of $f$ to $T$.
\end{proof}

Let $T$ and $T'$ be the input to the problem $d^2_{MP}$. Let $f_2, g_2$ and $T^{*}$ be the structures 
whose existence is guaranteed
by Lemma \ref{Lemma 2 taxon ccc}. Consider the forest obtained by cutting mutation edges (with respect to $g_2$) in $T^{*}$. Let $f_c$ be the character (i.e. partition of $X$) induced by the connected components of the forest. By the same lemma, each state of $f_c$ contains at least two taxa. Moreover, by construction, $f_c$ is convex on $T^{*}$. Next, we have a crucial lemma.

\begin{Lemma} \label{Lemma partial is covering}

    The natural partial extension of $f_c$ to $T^{*}$ is a natural covering extension.
\end{Lemma}

\begin{proof}
    Recall that $f_c$ was constructed from the optimal extension $g_2$ of the optimal two state character $f_2$. We know from Lemma \ref{Lemma internal node neighbor color} that, in $g_2$, every internal node $u$ of $T^{*}$ must have at least two neighbours with the same colour as $u$.  Thus, in the forest obtained from $g_2$, every internal node must have degree 2 or 3 (and every component of the forest contains at least 2 taxa). Furthermore we also know from Lemma \ref{Lemma 2 taxon ccc} that every leaf in a component of the forest is labelled by a taxon. This means that for every internal node $u$ of $T^{*}$, there exist two distinct taxa $x, y \in X$ such that $x$ and $y$ are assigned the same state by $f_c$, and $u$ lies on the unique path from $x$ to $y$ in $T^{*}$. Given that the natural partial extension was created using minimal spanning trees, it follows that the natural partial extension assigns a state to every internal node of $T^{*}$, and thus is a natural covering extension.   
\end{proof}

From Lemma \ref{strongUnique} and \ref{Lemma partial is covering} we obtain the following corollary.

\begin{Corollary} \label{Corollary unique extension}
    The convex character $f_c$ has a \emph{unique} optimal extension on $T^{*}$, which is the natural covering extension.
\end{Corollary}





In the remainder of the paper we shall call convex characters where the natural covering extension exists (which by Lemma \ref{strongUnique} are exactly those characters with unique optimal extensions) \emph{covering convex characters}. We have shown that $f_c$ is such a character\footnote{Note that not all convex characters
are covering. For example, consider the unique phylogenetic
tree on 3 taxa. If each of these three taxa is assigned
a different state by a character, the single interior node of the tree will be uncovered by the natural partial extension. The 8 convex characters shown in Figure \ref{fig:g2} are, however, covering.}. Now we have everything to prove our main theorem.

\begin{Theorem} \label{Theorem two state char is related to ccc}
    Let $T$ and $T'$ be the input to the problem $d^2_{MP}$. There exists a covering convex character $f_c$ on one of the trees, say $T$, such that (i) each state of $f_c$ contains at least 2 taxa and (ii) the states of $f_c$ can in polynomial time be relabelled red and blue, obtaining  a 2 state character $f_2$, such that $d_{MP}^2(T,T') = |l_{f_2}(T) - l_{f_2}(T')| = l_{f_2}(T') - l_{f_2}(T)$.
\end{Theorem}

\begin{proof}
    The earlier lemmas and corollaries prove the existence of a covering convex character $f_c$ on $T^{*} \in \{T,T'\}$. We know that (i) holds, by construction of $f_c$. Now, we know that the desired character $f_2$ exists - because we used it (together with $g_2$) to build $f_c$ in the first place - but to satisfy (ii) we have to show how to efficiently obtain $f_2$ given only $f_c$. We show (up to symmetry of red and blue) that $f_2$ is completely determined by $f_c$, and can easily be obtained from it.

    Crucially, due to being a covering convex character, there is a unique optimal extension of $f_c$ to $T^{*}$, which is the natural covering extension. Due to this uniqueness, any standard polynomial-time algorithm used for computing optimal extensions (e.g. Fitch's algorithm \cite{Fitch71}) will obtain it; let us call this $g_c$. This natural covering extension has, by construction, exactly the same mutation edges as $f_2$. Hence, $f_c$ partitions $X$ exactly the same way as the deletion of mutation edges in $g_2$ partitions $X$. We now need to decide which states of $f_c$ will become blue, and which will become red. 
 
   We define a new graph $G=(V,E)$ as follows.
    \begin{itemize}
        \item[$V$:] For every connected component in the forest obtained by cutting mutation edges of $g_c$ in $T^{*}$, we add a vertex to $V$ representing the connected component.
        \item[$E$:] For every mutation edge $\{u_1, u_2\}$ in the optimal extension $g_c$, we add the edge $\{v_1, v_2\}$ where the endpoints $v_1$ and $v_2$ are the connected components that $u_1$ and respectively $u_2$ belong to.
    \end{itemize}
    Note $G$ is still a tree; if it had cycles then so would $T^*$. Recall from
    Corollary \ref{Corollary unique extension} that $g_c$ is unique. This means that $G$ is also unique. Because $G$ is a connected tree we can see it as a bipartite graph with a unique bipartition. This bipartition of $X$ forms a two-state character. And because of the uniqueness of $G$ the mapping from $f_c$ to a two-state character is well-defined.
    

 Finally, we note that the construction of $G$ and the creation of the two-state character can be performed in polynomial time. This concludes the proof.

\end{proof}

Now that Theorem \ref{Theorem two state char is related to ccc} has been proven it is time to use it algorithmically. To compute $d_{MP}^2(T,T')$ we loop over all  covering convex characters $f_c$ of $T$, derive $f_2$ from it, compute $|l_{f_2}(T) - l_{f_2}(T')|$ and note the largest value we see. We then symmetrically repeat the procedure for $T'$, simply because we don't know whether $T^{*}$ is $T$ or $T'$. The character $f_2$ producing the largest value will be an optimal character.

All that remains is producing an algorithm that can enumerate all covering convex characters $f_c$ of a tree $T$. An elementary approach is to loop through the space of all convex characters where each state has at least two taxa, and discard those that are not covering convex characters. We can test in polynomial time whether a convex character is a covering convex character by constructing the natural partial extension and checking whether it covers all nodes of $T$. Kelk and Stamoulis \cite{KelkS17} proved that there are $\Theta(\phi^n)$ convex characters that have at least 2 taxa in each state and they can be efficiently listed. Putting all these pieces together, we end up with a running time of $O^{*}(\phi^n )$ for calculating $d_{MP}^2(T,T')$.

\begin{Theorem}
\label{firstAttempt}
Given two trees $T, T'$ on a set $X$ of $n$ taxa, $d_{MP}^2(T,T')$ can be computed in time $O^{*}(\phi^n )$.     
\end{Theorem}

But we can do better! In the next section we will show that there are actually at most
$O(1.5895^n)$ covering convex characters with at least two taxa per state, and that these can be efficiently listed, immediately giving a faster algorithm for $d_{MP}^2$. To prove this, we will turn to the topic of matchings.



\section{Matchings on a core tree}\label{Sec Matching Theory}


We proved that $f_c$, from which we can construct an optimal two-state character for $d^2_{MP}(T,T')$, is a covering convex character on one of the input trees, and has at least two taxa in every state. The following observation shows that some characters with these properties can be safely ignored. This will ultimately allow us to improve the running time given in Theorem \ref{firstAttempt}.

Suppose, without loss of generality, that $f_c$ has a lower parsimony score on $T$ than $T'$. Suppose furthermore that $T$ contains two taxa $a$ and $b$ whereby the shortest path from $a$ to $b$ has exactly three edges, and that $f_c$ assigns $a$ and $b$ the same state - but that no other taxa are assigned this state. See Figure \ref{fig illegal island}. An optimal extension here will be the natural covering extension, so the two red edges in the figure will definitely be mutation edges in this extension. If $f_c$ is mapped back to $f_2$ in the way described in Theorem \ref{Theorem two state char is related to ccc}, $f_2$ will assign $a$ and $b$ the same colour (say, blue) and the two mutation edges will be mutations from blue to red. But suppose, similar to the arguments used in Lemma  \ref{Lemma 2 taxon ccc}, that we flip $a$ and $b$ both to red in $f_2$, obtaining a new two-state character $f'_2$. We have that $l_{f'_2}(T) \leq l_{f_2}(T)-2$, because the two mutation edges disappear, but we also have $l_{f'_2}(T') \geq l_{f_2}(T') - 2$, because changing the state of two taxa can reduce the parsimony score by at most 2. Hence $f'_2$ is also an optimal character for the $d^2_{MP}$ instance, and still obeys the additional criteria specified in Lemma \ref{Lemma 2 taxon ccc}. In the same way $f_c$ was obtained from $f_2$, we can now construct $f'_c$ from $f'_2$. An algorithm that can locate
$f'_c$ is thus also guaranteed to optimise $d^2_{MP}$. The important fact
is that in an optimal extension of $f'_c$, neither of the red
edges shown in Figure \ref{fig illegal island} will be mutation edges.


\begin{figure}[h!]
\begin{center}
\begin{tikzpicture}[scale=0.8]
    \vertex  (u) at (1,0) {};
    \vertex  (v) at (2,0) {};
	\vertex [label=below:$a$] (a) at (1,-1) {};
    \vertex [label=below:$b$] (b) at (2,-1) {};
	\draw [line width = 1pt]
	(u) edge (a)
	(v) edge (b)
	(u) edge (v)
	(u) edge [red] (0,0)
	(v) edge [red] (3,0);
	
\end{tikzpicture}
\end{center}
\caption{Red edges are mutating (i.e. bichromatic) edges. When solving $d^2_{MP}$ we can safely ignore covering convex characters whose natural covering extensions have this local `island' structure.}
\label{fig illegal island}
\end{figure}
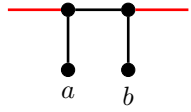

This observation shows that we need to find a way to skip over covering convex characters that have the kind of pointless `islands' described above and shown in Figure \ref{fig illegal island}. To do this we turn to \emph{matchings}. Recall that a matching of a graph is simply a subset of edges that do not share any endpoints. We will show a natural bijective correspondence between matchings (on an appropriately transformed tree) with covering convex characters that have at least two taxa in each state. 

\begin{Definition}
    The core tree $T_c$ for a phylogenetic tree $T$ on $X$ is the non-phylogenetic tree that remains when every node labelled by $X$ is removed.
\end{Definition}

\begin{Lemma}
There is a one-to-one correspondence between matchings in $T_c$ and covering convex characters of $T$ with at least two taxa per state.

\end{Lemma}

\begin{proof}
    Let $f_c$ be a covering convex character on $T$ with at least two taxa per state. We map $f_c$ to the set of  mutation edges in the unique optimal extension of $f_c$ to $T$. Note that this mapping is well-defined, since the optimal extension (and thus the set of induced mutation edges) is unique.
    Observe that none of these mutation edges can be incident to an element of $X$, because there are at least two taxa per state in $f_c$, so all these edges survive in $T_c$. Now, suppose the mutation edges do not form a matching in $T_c$. This means that there are two mutation edges that share an endpoint in $T_c$, and thus also in $T$. This is in contradiction with Lemma \ref{Lemma internal node neighbor color}. Hence, the mutation edges must form a matching in $T_c$. To see that the mapping from $f_c$ to matchings in $T_c$ is injective, 
    suppose $f_c$ and $f'_c$ are two distinct covering convex characters on $T$ i.e. they induce different partitions of $X$. Then there exist $x \neq y \in X$ such that without loss of generality $x,y$ belong to the same state $s$ in $f_c$ but to different states in $f'_c$. In the unique optimal extension of $f_c$ to $T$, all edges on the path from $x$ to $y$ must be covered by the spanning tree for $s$. However, in the unique optimal extension of $f'_c$ to $T$, at least one edge on this path does not lie on any of the spanning trees induced by the states of $f'_c$ - otherwise $x,y$ would belong to the same state\footnote{Recall that the spanning trees will be vertex-disjoint, because the tree has maximum degree 3, so it cannot occur that the path is the disjoint union of several spanning trees from distinct states.}. Hence the subset of edges of $T$ that are covered by spanning trees in the unique optimal extension of $f_c$ to $T$ is not the same subset as those corresponding to $f'_c$. The set of mutation (and thus matching) edges is exactly those \emph{not} on spanning trees, and as noted earlier all these mutation edges survive in $T_c$, so $f_c$ and $f'_c$ induce different matchings in $T_c$. Hence, injectivity holds.
    
    In the other direction, consider a matching $M$ of $T_c$. Observe that degree-1 nodes in $T_c$ correspond to internal nodes of $T$ that have two taxa children, and degree-2 nodes in $T_c$ to internal nodes of $T$ that have one taxon child. Suppose we delete the edges in $M$ from $T_c$, obtaining the forest $T_c - M$. We map $M$ to the natural partition of $X$ induced in $T$ by the connected components of $T_c - M$, i.e. each node in $T_c - M$ contributes the taxa it is adjacent to in $T$. This mapping is clearly well-defined. If we can show that each connected component of $T_c - M$ contains at least one degree-1 node (or at least two degree-2 nodes), where here the degrees refer to $T_c$ before the deletion of $M$, it follows that the mapping produces a partition of $X$ whereby each block contains at least two taxa. Suppose a connected component consists only of a single node. Then it must be a degree-1 node, and thus correspond to 2 taxa, so we are fine. (It cannot be a degree-2 or degree-3 node, because then the component would contain at least one edge). So suppose a connected component contains at least one edge, and thus has at least two leaves. Note that leaves of a connected component of $T_c - M$ \emph{cannot} be degree-3 nodes in $T_c$, because this would mean that the connected component could actually be grown further: this is because there can be at most one matching edge incident to a degree-3 node. So each of the at least two leaves must be a degree-1 or degree-2 node. Hence, we indeed obtain a partition of $X$ with at least two taxa per block. To see that this partition is in fact a covering convex character, observe that the connected components of $T_c - M$ are (when projected onto $T$) minimal spanning trees for the blocks of $X$. They are minimal because, as noted above, the
    leaves of components in $T_c - M$ must be degree-1 or degree-2 nodes in $T_c$, and such nodes are (in $T$) incident to taxa. The minimal spanning trees cover all nodes of $T$, because the connected components of $T_c - M$ cover all nodes of $T_c$. From this we conclude that the partition of $X$ that we created is indeed a covering convex character (and that the edges in $M$ are precisely the mutation edges in the unique optimal extension for this character).
    
    Finally, to see that the mapping from matchings to characters is injective, suppose two matchings $M, M'$ of $T_c$ map to the same covering convex character $f_c$ with at least two taxa per state. As stated above, $M$ and $M'$ are mutation edges in an optimal extension of $f_c$ to $T$. By Corollary  \ref{Corollary unique extension} a covering convex character has a unique optimal extension, so $M=M'$. 

\end{proof}




\begin{Definition}
    We call a matching $M$ in $T_c$ \emph{legal} if $T_c - M$ does not contain a component containing a single edge connected to two edges in $M$. Otherwise it is an \emph{illegal} matching.
\end{Definition}

Observe that illegal matchings induce $f_c$ with the `islands' discussed at the start of the section, and shown in  Figure \ref{fig illegal island}. We know it does not compromise optimality to ignore such characters. Hence, in solving $d^2_{MP}$ we can henceforth focus on counting and listing only legal matchings.

What is an upper bound for the number of legal matchings in $T_c$, and thus an upper bound for the number of candidate $f_c$? To answer this question we establish a recursion, which we will also use to create a DP algorithm to enumerate all the legal matchings. When bounding the recursion we will require some auxiliary support from the software package Mathematica; we shall explain this further in due course.

To calculate the upper bound, we apply the method described in \cite{Rosenfeld21} (and earlier in \cite{Rote19}). To this end, we need to introduce some additional notation.

In the remainder of this section, to ensure relevance to the wider matchings literature, $T$ will refer to any non-phylogenetic tree which has maximum degree 3. (In our specific $d^2_{MP}$ applications context, $T$ will be the tree $T_c$ created
earlier in the section, but this extra information is not necessary.) Again, to keep consistency with the matchings literature, we will let $n$ henceforth refer to the total number of nodes in $T$, not just the number of leaves. Fortunately, this will not asymptotically distort any running times when we translate the results back to the phylogenetic context. That is because a phylogenetic tree with $|X|$ leaves has exactly $|X|-2$ internal nodes, so $T_c$ (where we enumerate legal matchings) has exactly $|X|-2$ nodes in total.

We consider all trees as \emph{rooted} trees of maximum degree 3, where each node has at most one left child and at most one right child. We select an arbitrary leaf of the tree as the root; this allows us to avoid certain technicalities later on.


It will be convenient to allow empty trees with no nodes in the following. Define 
 \[
e(T) = [T \text{ is empty}] = \begin{cases} 1 & T \text{ is an empty tree,} \\ 0 & \text{otherwise.} \end{cases}
\]

Moreover, we consider two special types of matchings of a rooted tree $T$: 
\begin{itemize}
\item Legal matchings where the root has only one child and only one grandchild, the root is not covered by a matching edge, but its child is.   The number of these matchings is denoted by $b_0(T)$ and shown in Figure \ref{Fig legal matching origins b_0}.

\begin{figure}[!h]
\begin{center}
\begin{tikzpicture}[scale = 0.8]	
	\vertex [label=left:$u$] (u) at (0,0) {};
	\vertex [label=left:$c$] (c) at (0,-1) {};
	\vertex [label=left:$gc$] (gc) at (0,-2) {};
	
	\draw [line width = 1pt]
	(u) edge (c)
	(c) edge [red] (gc)
	(u) edge (0,1)	
    (gc) -- (-1,-3) -- (1,-3) -- (gc);
	
\end{tikzpicture}
\end{center}
\caption{Legal matchings stored in $b_0$. A red edge represents an edge in the matching.
As in the other figures in this section the edge entering from above represents the edge that enters the subtree in the original tree.}
\label{Fig legal matching origins b_0}
\end{figure}
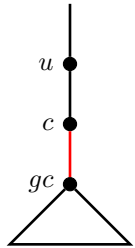

\item Legal matchings where the root has one child only and is covered by a matching edge. Their number is denoted by $b_1(T)$ and shown in Figure \ref{Fig legal matching origins b_1}.

\begin{figure}[!h]
\begin{center}
\begin{tikzpicture}[scale = 0.8]	
	\vertex [label=left:$u$] (u) at (0,0) {};
	\vertex [label=left:$c$] (c) at (0,-1) {};
	
	\draw [line width = 1pt]
	(u) edge [red] (c)
	(u) edge (0,1)
    (c) -- (-1,-2) -- (1,-2) -- (c);
	
\end{tikzpicture}
\end{center}
\caption{Legal matchings stored in $b_1$. A red edge represents an edge in the matching.}
\label{Fig legal matching origins b_1}
\end{figure}
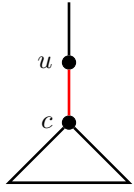

\end{itemize}

We denote the number of remaining legal matchings not covering the root by $a_0(T)$, shown in Figure \ref{Fig legal matching origins a_0}. \begin{figure}[!h]
\begin{center}
\begin{tikzpicture}[scale = 0.8]
	\vertex [label=left:$u$] (u) at (0,0) {};
	\vertex [label=left:$l$] (l) at (-1.5,-1) {};
    \vertex [label=right:$r$] (r) at (1.5,-1) {};
	
	\draw [line width = 1pt]
	(u) edge (l)
	(u) edge (r)
    (u) edge (0,1)
    (l) -- (-2.5,-2) -- (-0.5,-2) -- (l)
    (r) -- (0.5,-2) -- (2.5,-2) -- (r);
	
	\vertex [label=left:$u$] (u) at (4.5,0) {};
	\vertex [label=left:$l$] (c) at (4.5,-1) {};
	
	\draw [line width = 1pt]
	(u) edge (c)
	(u) edge (4.5,1)
    (c) -- (3.5,-2) -- (5.5,-2) -- (c);
	
\end{tikzpicture}
\end{center}
\caption{Legal matchings stored in $a_0$.}
\label{Fig legal matching origins a_0}
\end{figure}
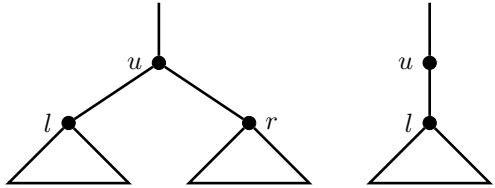

Finally, we let $a_1(T)$ be the number of remaining legal matchings covering the root, shown in Figure \ref{Fig legal matching origins a_1}.

\begin{figure}[!h]
\begin{center}
\begin{tikzpicture}[scale = 0.8]
	\vertex [label=left:$u$] (u) at (0,0) {};
	\vertex [label=left:$l$] (l) at (-1.5,-1) {};
    \vertex [label=right:$r$] (r) at (1.5,-1) {};
	
	\draw [line width = 1pt]
	(u) edge [red] (l)
	(u) edge (r)
    (u) edge (0,1)
    (l) -- (-2.5,-2) -- (-0.5,-2) -- (l)
    (r) -- (0.5,-2) -- (2.5,-2) -- (r);
    
    \vertex [label=left:$u$] (u) at (6,0) {};
	\vertex [label=left:$l$] (l) at (4.5,-1) {};
    \vertex [label=right:$r$] (r) at (7.5,-1) {};
	
	\draw [line width = 1pt]
	(u) edge (l)
	(u) edge [red] (r)
    (u) edge (6,1)
    (l) -- (3.5,-2) -- (5.5,-2) -- (l)
    (r) -- (6.5,-2) -- (8.5,-2) -- (r);
	
\end{tikzpicture}
\end{center}
\caption{Legal matchings stored in $a_1$. A red edge represents an edge in the matching.}
\label{Fig legal matching origins a_1}
\end{figure}
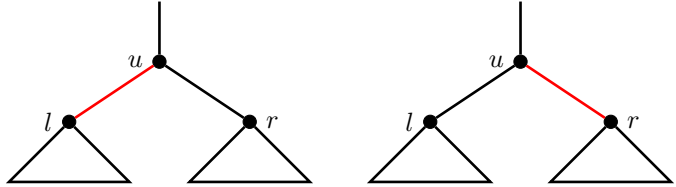

If $T$ is empty, we set $a_0(T) = a_1(T) = b_0(T) = b_1(T) = 0$. Let $T$ be a rooted tree whose root branches are $T_l$ and $T_r$ (possibly empty). More formally, $T_l$ (respectively, $T_r$) is the subtree of $T$ rooted at the left (respectively, right) child of $T$. In the following recursions, whereby $T$ can thus represent a subtree of $T$ - as opposed to the original tree itself - we emphasise that $a_0, a_1, b_0, b_1$ count matchings that are legal on the \emph{original} tree, not just the subtree. This is important because although subtrees $T_l, T_r$ encountered in the recursion can potentially have a root of degree 1 or 2 (or even degree 0), the subtree will have an incoming edge when viewed in the context of the original tree. This ensures that some matchings which would be illegal if the subtree was considered in isolation, but which are legal when viewed in the context of the original tree, are correctly counted. 
\begin{align*}
a_0(T) &= \big(a_0(T_l) + a_1(T_l) + b_0(T_l) + b_1(T_l) + e(T_l)\big)\big(a_0(T_r) + a_1(T_r) + b_0(T_r) + b_1(T_r) + e(T_r)\big) \\
&\qquad - b_1(T_l)e(T_r) - e(T_l)b_1(T_r), \\
a_1(T) &= a_0(T_l)\big(a_0(T_r) + a_1(T_r) + b_0(T_r) + b_1(T_r)\big) + \big(a_0(T_l) + a_1(T_l) + b_0(T_l) + b_1(T_l) \big)a_0(T_r), \\
b_0(T) &= b_1(T_l)e(T_r) + e(T_l)b_1(T_r), \\
b_1(T) &= a_0(T_l)e(T_r) + e(T_l)a_0(T_r).
\end{align*}
Note that, if $T$ is a single node, $a_0(T)$ evaluates to 1 (corresponding to the empty matching), but $a_1(T)$, $b_0(T)$ and $b_1(T)$ all
evaluate to 0.
The reasoning behind the recursion is as follows. We obtain a matching not covering the root in $T$ by combining an arbitrary matching in $T_l$ (possibly the empty matching if $T_l$ is empty) and an arbitrary matching in $T_r$ (also possibly empty). If both components are legal, then so is the new matching. It is a special matching counted by $b_0$ if one of the two branches is empty and the matching chosen in the other branch is a special matching counted by $b_1$.

Likewise, we obtain a matching covering the root by first choosing one of the root edges, e.g.~the edge between the root of $T$ and the root of $T_l$. Then we combine it with an arbitrary matching in $T_r$ and a matching not covering the root of $T_l$ (or vice versa if the other root edge was chosen). For the matching to remain legal, the matching in $T_l$ may not be of special type $b_0$. A special matching $b_1$ is obtained in this way if one branch is empty. 


It is useful to write this recursion in matrix form: associating a vector
$$\mathbf{v}(T) = \big[a_0(T),a_1(T),b_0(T),b_1(T),e(T)\big]^T$$ 
to a tree $T$, we have $\mathbf{v}(T)  = B\big(\mathbf{v}(T_l),\mathbf{v}(T_r)\big)$, where the map $B: \R^5 \times \R^5 \to \R^5$ is defined by
\[
B\big(
\begin{bmatrix} 
    v_1 \\ 
    w_1 \\ 
    x_1 \\ 
    y_1 \\ 
    z_1 
\end{bmatrix},
\begin{bmatrix} 
    v_2 \\
    w_2 \\
    x_2 \\
    y_2 \\
    z_2
\end{bmatrix}
\big)
= 
\begin{bmatrix} 
(v_1+w_1+x_1+y_1+z_1)(v_2+w_2+x_2+y_2+z_2)-y_1z_2-z_1y_2 \\
(v_1+w_1+x_1+y_1)v_2+v_1(v_2+w_2+x_2+y_2) \\
y_1z_2+z_1y_2 \\
v_1z_2+z_1v_2 \\
0
\end{bmatrix}.
\]
Note that $B$ is a bilinear map: we have 
\[
B(\mathbf{v}_1 + \mathbf{w}_1, \mathbf{v}_2 + \mathbf{w}_2) = B(\mathbf{v}_1, \mathbf{v}_2) + B(\mathbf{w}_1, \mathbf{v}_2) + B(\mathbf{v}_1,\mathbf{w}_2) +  B(\mathbf{w}_1,\mathbf{w}_2)
\]
and
\[
B(c_1 \mathbf{v}_1, c_2 \mathbf{v}_2) = c_1c_2 B(\mathbf{v}_1, \mathbf{v}_2).
\]
Furthermore the vector associated with the empty tree is $[0,0,0,0,1]^T$.

\begin{Theorem}\label{thm:upperbound}
The maximum number $M_n$ of legal matchings in a tree with $n$ nodes is $O(\alpha^n)$ with $\alpha = (13384+8\sqrt{2793745})^{1/22} \approx 1.5895$. 
\end{Theorem}

\begin{proof}
There exists a set $\mathcal{S}$ of $62$ $5$-dimensional vectors with nonnegative entries that have the following property:
\begin{itemize}
\item[(1)] It contains the vector $[0,0,0,0,1/\alpha]^T$,
\item[(2)] For any pair of two vectors $\mathbf{v}_1,\mathbf{v}_2 \in \mathcal{S}$, the vector $B(\mathbf{v}_1,\mathbf{v}_2)$ lies in the set
$$conv_{\leq}(\mathcal{S}) = \big\{ \mathbf{w} \in \R^5\,:\, \mathbf{w} \geq \mathbf{0}, \mathbf{w} \leq \sum_{\mathbf{v} \in \mathcal{S}} c_{\mathbf{v}} \mathbf{v} \text{ for some constants } c_{\mathbf{v}} \geq 0, \sum_{\mathbf{v} \in \mathcal{S}} c_{\mathbf{v}} = 1 \big\}.$$
\end{itemize}
Here, the inequalities hold componentwise. Note that $conv_{\leq}(\mathcal{S})$ is a bounded and convex set by construction. The set $\mathcal{S}$ is listed in Appendix B.  It can be verified directly (with the help of a computer) that the second condition is indeed satisfied. For this purpose,
we have provided a comprehensively annotated Mathematica file in the supplementary material.

\medskip

Given this property of $\mathcal{S}$, we can now prove the following by induction on $n$: for every rooted binary tree $T$ with $n$ nodes, the vector $\alpha^{-n-1} \mathbf{v}(T)$ lies in $conv_{\leq}(\mathcal{S})$. This is trivial for $n=0$, since we get the vector $[0,0,0,0,1/\alpha]^T$ for the empty tree, which lies in $\mathcal{S}$ by property (1) and thus in turn in $conv_{\leq}(\mathcal{S})$. For the induction step, we can apply property (2) of $\mathcal{S}$. Assume that the two branches $T_1$ and $T_2$ (possibly empty) of $T$ satisfy the statement, and let them have $k$ and $n-k-1$ nodes respectively. We have
\[
\alpha^{-n-1} \mathbf{v}(T) = \alpha^{-n-1} B(\mathbf{v}_1,\mathbf{v}_2) = B (\alpha^{-k-1}\mathbf{v}_1,\alpha^{-n+k}\mathbf{v}_2).
\]
By the induction hypothesis, both $\alpha^{-k-1}\mathbf{v}_1$ and $\alpha^{-n+k}\mathbf{v}_2$ lie in $conv_{\leq}(\mathcal{S})$, so there exist linear combinations of the elements of $\mathcal{S}$ with nonnegative coefficients such that
\[
\alpha^{-k-1}\mathbf{v}_1 \leq \sum_{\mathbf{v} \in \mathcal{S}} c^{(1)}_{\mathbf{v}}\mathbf{v}
\]
and
\[
\alpha^{-n+k}\mathbf{v}_2 \leq \sum_{\mathbf{v} \in \mathcal{S}} c^{(2)}_{\mathbf{v}}\mathbf{v},
\]
thus by (bi-)linearity of $B$
\begin{align*}
\alpha^{-n-1} \mathbf{v}(T) &= B (\alpha^{-k-1}\mathbf{v}_1,\alpha^{-n+k}\mathbf{v}_2) \\
&\leq \sum_{\mathbf{v} \in \mathcal{S}}  \sum_{\mathbf{w} \in \mathcal{S}} c^{(1)}_{\mathbf{v}}
c^{(2)}_{\mathbf{w}} B(\mathbf{v},\mathbf{w}).
\end{align*}
Since $B(\mathbf{v},\mathbf{w}) \in conv_{\leq}(\mathcal{S})$ for all $\mathbf{v}$ and $\mathbf{w}$ and $conv_{\leq}(\mathcal{S})$ is convex, it follows that $\alpha^{-n-1} \mathbf{v}(T) \in conv_{\leq}(\mathcal{S})$, completing the induction.

In particular, we have shown that the entries of the vector $\alpha^{-n-1} \mathbf{v}(T)$ are bounded. The total number of legal matchings of $T$ is the sum of the entries of $\mathbf{v}(T)$, so it follows that this number is $O(\alpha^n)$.
\end{proof}

\noindent
\emph{Remark. }The set $\mathcal{S}$ in our proof was constructed by iteration as described in \cite{Rosenfeld21}, starting with the two vectors $[0,0,0,0,1/\alpha]^T$ and $[\alpha^{-7}(4 + 6616/\sqrt{2793745}),0,0,\alpha^{-7}(2 + 3446/\sqrt{2793745}),0]$. The choice of the two vectors is justified as follows: the first vector corresponds to the empty tree; the second stems (as a limit point) from the sequence of trees that yields the lower bound presented in the appendix.\\


Now that the upper bound on the number of legal matchings has been proved, all we need is a dynamic programming algorithm that allows us to enumerate every legal matching. As noted earlier, we take $T_c$ as the tree $T$. We have a recursion that allowed us to count all the legal matchings in $T$. Thus, just like the algorithm in section \ref{Sec uMAF}, we use this to calculate the exact number of legal matchings in $T_c$ and use the recursion to backtrack in $T_c$ to enumerate each legal matching. Specifically, observe that the terms $a_0, a_1, b_0, b_1$ together sum to the number of legal matchings, and that each of these terms is itself a summation of terms; note here that the negative terms in $a_0$ actually cancel when the product is expanded. Hence, it is straightforward to canonically index the legal matchings and to use this to steer the backtracking.

After that we can transform the matchings back to a corresponding covering convex character $f_c$ in $T$. The states can then be coloured red-blue in the fashion described earlier, to obtain all possible $f_2$, and we are done. This brings us to our main theorem:

\begin{Theorem}
    For a pair of binary trees $T$ and $T'$ on $|X|$ taxa we can compute
    $d_{MP}^2$ in $O(\alpha^{|X|} \cdot \text{poly}(|X|))$ time with $\alpha = (13384+8\sqrt{2793745})^{1/22} \approx 1.5895$. 
\end{Theorem}

\begin{proof}
    Theorem \ref{Theorem two state char is related to ccc} proves that the process of enumerating all covering characters $f_c$ on $T$ with at least two taxa per state will give us a $f_2$ such that $d_{MP}^2(T,T') = |l_{f_2}(T) - l_{f_2}(T')| = l_{f_2}(T') - l_{f_2}(T)$, when $T$ has the lower parsimony score at optimality. Kelk and Fischer \cite{KelkF17} proved that $d_{MP}^2(T,T')$ is asymmetrical, in the sense that, if we range over all two-state characters $f$, the maximum value of $l_f(T')-l_f(T)$ might differ from
    the maximum of $l_f(T)-l_f(T')$.
    This means we don't know \emph{a priori} which of $T$ and $T'$ has the lowest parsimony score at optimality. Therefore we simply enumerate the $f_c$ of both trees. This has no impact on the asymptotic runtime.
    
    The process of creating $T_{c}$ (and in a second iteration, $T'_{c})$ can be achieved in polynomial time. The combination step in the dynamic programming algorithm - where we use the values of two child nodes to derive the values of the parent node - can be done in constant time and thus the whole dynamic programming algorithm is clearly polynomial. A legal matching can be translated to a covering convex character $f_c$ in polynomial time. The steps taken for each $f_c$ to create $f_2$ can also be done in polynomial time.
    Optimal extensions, and thus the parsimony score, can be computed in polynomial time so each $|l_{f_2}(T) - l_{f_2}(T')|$ is also polynomial-time computable.
    Theorem \ref{thm:upperbound} proves that there are at most $O(\alpha^n)$ legal matchings in $T_c$, where in that theorem $n$ refers to the \emph{total} number of nodes in $T_c$. Here $|X|$ refers to the number of leaves/taxa in $T$. However, as observed earlier, $n = |X|-2$, so the asymptotics do not change in the translation.
    
    Hence, the dynamic programming algorithm will enumerate at most $O(\alpha^{|X|})$ natural covering extensions.
    
    Combining everything gives the final runtime $O(\alpha^{|X|} \cdot \text{poly}(|X|))$.
\end{proof}

\section{Discussion: lower bounds, fully legal matchings, and a new upper bound for matchings on restricted trees}
\label{sec:discussion}

In this section we prove a number of extensions of the upper bound result from the previous section. 

\subsection{A matching lower bound for legal matchings}

First, we show in Lemma \ref{lem:lb} in Appendix C that there is a lower bound matching the upper bound in Theorem \ref{thm:upperbound}. From this we
conclude that, ranging over trees with $n$ nodes and at most degree 3, the maximum number of legal matchings is not just $O(\alpha^n)$ but in fact $\Theta(\alpha^n)$,
where $\alpha = (13384+8\sqrt{2793745})^{1/22} \approx 1.5895$. 

\subsection{A new upper bound on the number of matchings in a restricted subfamily of trees}

Next, we translate the upper bound on the number of legal matchings in trees of maximum degree at most 3 to an upper bound on the number of matchings in \emph{good} trees, henceforth defined as trees of degree at most 3, where no two degree 2 nodes are adjacent. To prove this we first need the following lemma.

\begin{Lemma}\label{lem:notwoadj}
For every tree $T$ with maximum degree at most $3$, there is a tree $T^*$ with maximum degree at most $3$ and no two adjacent nodes of degree $2$ (i.e. a good tree) that has more or equally many legal matchings.
\end{Lemma}

\begin{proof}
Suppose that $v_1$ and $v_2$ are two adjacent nodes of degree $2$ in $T$, and let their other neighbours be $v_0$ and $v_3$ respectively. Let the edges $v_0v_1$, $v_1v_2$ and $v_2v_3$ be denoted $e_0, e_1, e_2$ respectively. We remove the edge $e_0$ and replace it by $e_0' = v_0v_2$ to obtain a new tree $T'$. Now let $\Phi$ map $e_0$ to $e_0'$ and all other edges to themselves. We show that $\Phi(M)$ is a legal matching of $T'$ whenever $M$ is a legal matching of $T$. Note first that the only edges that could be adjacent in $\Phi(M)$ (given that there are no adjacent edges in $M$) are $e_0'$ and $e_2$. However, if both $e_0$ and $e_2$ are in $M$, then this would contradict the assumed legality of $M$. So $\Phi(M)$ is indeed a matching. Next, observe that the legality of $M$ implies that at most one of $e_0, e_1, e_2$ are in $M$, and that in $T'$,
$v_1$ and $v_2$ no longer have degree 2. Hence, if $\Phi(M)$ was not a legal matching, it could only be due to $\Phi(M)$ containing a matching edge incident to $v_0$ or $v_3$. But then $M$ was itself not legal. Hence, $\Phi(M)$ is indeed legal. As $\Phi$ is injective, $T'$ has at least as many legal matchings as $T$. Iterating the construction, we end up with a tree that has no adjacent nodes of degree $2$.
\end{proof}



Lemma \ref{lem:notwoadj} shows that, among trees with $n$ nodes and of maximum degree 3,
the number of legal matchings is maximised by some good tree. Moreover, by construction
it follows that \emph{every} matching in a good tree is a legal matching. The following
result then follows immediately from Theorem \ref{thm:upperbound}.

\begin{Theorem}
\label{thm:newbound}
A tree with $n$ nodes and maximum degree 3, where no two degree 2 nodes are adjacent (i.e. a good tree), has $O(\alpha^n)$ matchings, where $\alpha = (13384+8\sqrt{2793745})^{1/22} \approx 1.5895$.
\end{Theorem}

Theorem \ref{thm:newbound} shows that the absence of adjacent degree 2 nodes is sufficient to slightly improve the well-known upper bound of $O(\phi^n)$ matchings in trees. The lower bound construction given in Lemma \ref{lem:lb} in the appendix constructs good trees that asymptotically attain the bound of Theorem~\ref{thm:newbound}, so $\alpha$ cannot be lowered any further for good trees. We note that for trees where every internal node has degree 3, there are $O(1.5538^n)$ matchings
(see \cite[Theorem 1 and Remark 5]{andriantiana2011number}; the precise constant is $\sqrt{1+\sqrt{2}}$)). 




\subsection{Strengthening legality: full legality}

When computing $d^2_{MP}$ we could restrict our attention to legal matchings in the core tree $T_c$, because we showed that there exists at least one character $f_c$ that optimises $d^2_{MP}$, such that $f_c$ is induced by a legal matching. The central idea was that an \emph{illegal} matching induces a covering convex character $f$ whereby, in its unique optimal extension, an island comprising two taxa is incident to two mutations. By flipping the state assigned by the character to the two taxa, we could obtain a new covering convex character $f_c$ without the island, such that $|l_{f_c}(T')-l_{f_c}(T)| \geq |l_{f}(T')-l_{f}(T)|$. 

We can extend this idea to  bigger islands. This requires the introduction of weights to the nodes of $T_c$. The weight represents the number of taxa a node is connected to in $T$. Now, consider a matching $M$ of $T_c$. Let $I$ be a maximal, connected subtree of $T_c$ that does not include an edge of $M$. Suppose the total weight $w$ of $I$ is less than or equal to the number of edges of $M$ it is incident to. In this case, the state flipping argument described above will go through. (Specifically, flipping the state saves at least $w$ mutations in the tree with lower parsimony score, but cannot decrease the parsimony score of the other tree by more than $w$.) This state flipping could be applied, for example, in the situation described in Figure \ref{fig big illegal island}. 

Once all such situations have been excluded, we see that when computing $d^2_{MP}$ we can safely restrict our attention to matchings where every region enclosed between matching edges has total weight strictly larger than the number of incident matching edges.

\begin{figure}
\begin{center}
\begin{tikzpicture}[scale=0.8]
    \vertex  (u) at (0,0) {};
    \vertex  (v) at (1,0) {};
    \vertex  (uu) at (-1,1) {};
    \vertex  (ub) at (-1,-1) {};
    \vertex  (vu) at (2,1) {};
    \vertex  (vb) at (2,-1) {};
    
	\vertex [label=below:$a$] (a) at (-2,1) {};
    \vertex [label=below:$b$] (b) at (-2,-1) {};
    \vertex [label=below:$c$] (c) at (3,1) {};
    \vertex [label=below:$d$] (d) at (3,-1) {};

	\draw [line width = 1pt]
	(u) edge (uu)
	(u) edge (ub)
	(v) edge (vu)
	(v) edge (vb)
	(u) edge (v)
	
	(uu) edge (a)
	(ub) edge (b)
	(vu) edge (c)
	(vb) edge (d)
	
	(uu) edge [red] (-1,2)
	(ub) edge [red] (-1,-2)
	(vu) edge [red] (2,2)
	(vb) edge [red] (2,-2);
	
	\draw [line width = 1pt]
	(4.5,0) edge [->] node [label=above:$Construct\ T_c$] {} (7.5,0);

    \vertex  [label=below:$0$] (u) at (10,0) {};
    \vertex  [label=below:$0$] (v) at (11,0) {};
    \vertex  [label=left:$1$] (uu) at (9,1) {};
    \vertex  [label=left:$1$] (ub) at (9,-1) {};
    \vertex  [label=right:$1$] (vu) at (12,1) {};
    \vertex  [label=right:$1$] (vb) at (12,-1) {};

	\draw [line width = 1pt]
	(u) edge (uu)
	(u) edge (ub)
	(v) edge (vu)
	(v) edge (vb)
	(u) edge (v)
	
	(uu) edge [red] (9,2)
	(ub) edge [red] (9,-2)
	(vu) edge [red] (12,2)
	(vb) edge [red] (12,-2);
	
\end{tikzpicture}
\end{center}
\caption{A red edge represents an edge in the matching. This cannot be a fully legal matching, because in $T_c$ the island enclosed by the 4 matching edges has total weight 4. Hence, such a matching can safely be ignored when optimising $d^2_{MP}$.}
\label{fig big illegal island}
\end{figure}
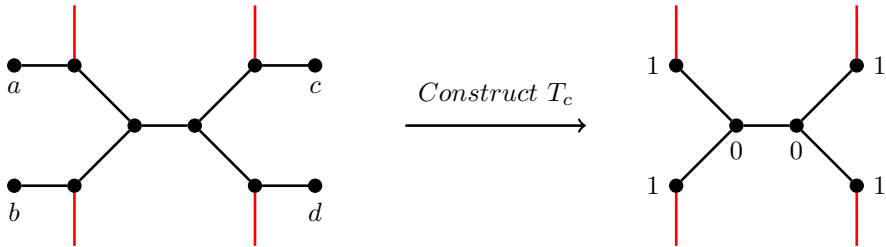

The following two definitions formalise this idea.

\begin{Definition}
Let $T$ be a binary phylogenetic tree on $X$ and $T_c$ the core tree of $T$. Every node of $T_c$ is assigned a weight that is equal to the number of leaves in $T$ it is adjacent to. Note that the weight of $v$ can also be expressed as $3 - \deg(v)$, where $deg(v)$ is the degree of $v$ in $T_{c}$.

A matching $M$ of $T_c$ is called $k$-legal if for every connected component $S$ of $T_c - M$ with at most $k$ vertices, the  total weight of $S$ (i.e., the sum of the weights of all its nodes) is strictly greater than the number of edges in $M$ that are incident to $S$.
\end{Definition}

Note that every matching $M$ of $T_c$ is a 0-legal and 1-legal matching because every component in $T_c - M$ has at least two vertices. With this new definition we can translate the definition of legal, as used earlier in the article, to 2-legal. The only way for a matching to not be 2-legal is for it to contain a component shown in Figure \ref{fig illegal island}.

\begin{Definition}
A matching $M$ of $T_c$ is called fully legal if it is $k$-legal for every value of $k$. In other words, for every component $S$ of $T_c - M$, the total weight of $S$ is strictly greater than the number of edges in $M$ that are incident to $S$. Note that if $M$ is $n$-legal with $n$ the number of leaves in $X$, $M$ is automatically fully legal.
\end{Definition}

Let $S$ be a component of $T_c - M$ for some matching $M$, let $s = |S|$ be the number of vertices and $m$ the number of edges in $M$ incident to $S$. The sum of the degrees of the vertices in $S$ is $2(s-1) + m$, as edges between vertices of $S$ are counted twice, and incident edges of $M$ counted once. Thus the total weight of $S$ is $3s - (2(s-1)+m) = s-m+2$, which allows us to reformulate the definition of $k$-legality as follows:

\begin{Lemma}
A matching $M$ of $T_c$ is $k$-legal if and only if, for every connected component $S$ of $T_c - M$ with at most $k$ vertices, the number of vertices $s$ and the number of incident edges $m$ satisfy $s > 2m-2$.
\end{Lemma}

%
%
%
%
%
%
%
%
%
Theorem~\ref{thm:upperbound} only provides an upper bound on the number of 2-legal matchings. How far can this upper bound be lowered if we restrict our attention purely to fully legal matchings?

At this stage we only have partial answers. The greater $k$, the more complex the problem becomes, and the more auxiliary quantities are required. For $4$-legal matchings, one obtains a system of recurrences and in turn a $13$-dimensional bilinear map akin to the map $B$ in the proof of Theorem~\ref{thm:upperbound}. With the help of a computer - we provide a detailed Mathematica file in the supplementary material - one can then use the same method to show that the number of such matchings is $O((19/12)^n)$ in trees with $n$ nodes (this bound is, unlike Theorem~\ref{thm:upperbound}, probably not sharp). Note here that $19/12 \approx 1.5834  < \alpha \approx 1.5895$ - thus implying a marginally improved $O^{*}(1.5834^n)$  algorithm for $d^2_{MP}$. It follows that the number of fully legal matchings in trees with $n$ nodes is also $O((19/12)^n)$. Further improvements to this bound should be possible by considering larger values of $k$, but the resulting recursions become quite unwieldy.

Finally, we provide a lower bound. By means of a computer search, one finds that the trees with $15$ to $20$ nodes and the greatest number of fully legal matchings look as follows:

\begin{center}
\begin{tikzpicture}[scale=0.7]
\node[fill=black,circle,inner sep=1.5pt]  at (0,0) {};
\node[fill=black,circle,inner sep=1.5pt]  at (1,0) {};
\node[fill=black,circle,inner sep=1.5pt]  at (2,0) {};
\node[fill=black,circle,inner sep=1.5pt]  at (3,0) {};
\node[fill=black,circle,inner sep=1.5pt]  at (4,0) {};
\node[fill=black,circle,inner sep=1.5pt]  at (5,0) {};
\node[fill=black,circle,inner sep=1.5pt]  at (6,0) {};
\node[fill=black,circle,inner sep=1.5pt]  at (1,1) {};
\node[fill=black,circle,inner sep=1.5pt]  at (2,1) {};
\node[fill=black,circle,inner sep=1.5pt]  at (3,1) {};
\node[fill=black,circle,inner sep=1.5pt]  at (4,1) {};
\node[fill=black,circle,inner sep=1.5pt]  at (1,2) {};
\node[fill=black,circle,inner sep=1.5pt]  at (2,2) {};
\node[fill=black,circle,inner sep=1.5pt]  at (3,2) {};
\node[fill=black,circle,inner sep=1.5pt]  at (4,2) {};

\draw (0,0)--(6,0);
\draw (1,0)--(1,2);
\draw (2,0)--(2,2);
\draw (3,0)--(3,2);
\draw (4,0)--(4,2);

\node[fill=black,circle,inner sep=1.5pt]  at (8,0) {};
\node[fill=black,circle,inner sep=1.5pt]  at (9,0) {};
\node[fill=black,circle,inner sep=1.5pt]  at (10,0) {};
\node[fill=black,circle,inner sep=1.5pt]  at (11,0) {};
\node[fill=black,circle,inner sep=1.5pt]  at (12,0) {};
\node[fill=black,circle,inner sep=1.5pt]  at (13,0) {};
\node[fill=black,circle,inner sep=1.5pt]  at (14,0) {};
\node[fill=black,circle,inner sep=1.5pt]  at (15,0) {};
\node[fill=black,circle,inner sep=1.5pt]  at (10,1) {};
\node[fill=black,circle,inner sep=1.5pt]  at (11,1) {};
\node[fill=black,circle,inner sep=1.5pt]  at (12,1) {};
\node[fill=black,circle,inner sep=1.5pt]  at (13,1) {};
\node[fill=black,circle,inner sep=1.5pt]  at (10,2) {};
\node[fill=black,circle,inner sep=1.5pt]  at (11,2) {};
\node[fill=black,circle,inner sep=1.5pt]  at (12,2) {};
\node[fill=black,circle,inner sep=1.5pt]  at (13,2) {};

\draw (8,0)--(15,0);
\draw (10,0)--(10,2);
\draw (11,0)--(11,2);
\draw (12,0)--(12,2);
\draw (13,0)--(13,2);

\node[fill=black,circle,inner sep=1.5pt]  at (0,-4) {};
\node[fill=black,circle,inner sep=1.5pt]  at (1,-4) {};
\node[fill=black,circle,inner sep=1.5pt]  at (2,-4) {};
\node[fill=black,circle,inner sep=1.5pt]  at (3,-4) {};
\node[fill=black,circle,inner sep=1.5pt]  at (4,-4) {};
\node[fill=black,circle,inner sep=1.5pt]  at (5,-4) {};
\node[fill=black,circle,inner sep=1.5pt]  at (6,-4) {};
\node[fill=black,circle,inner sep=1.5pt]  at (1,-3) {};
\node[fill=black,circle,inner sep=1.5pt]  at (2,-3) {};
\node[fill=black,circle,inner sep=1.5pt]  at (3,-3) {};
\node[fill=black,circle,inner sep=1.5pt]  at (4,-3) {};
\node[fill=black,circle,inner sep=1.5pt]  at (5,-3) {};
\node[fill=black,circle,inner sep=1.5pt]  at (1,-2) {};
\node[fill=black,circle,inner sep=1.5pt]  at (2,-2) {};
\node[fill=black,circle,inner sep=1.5pt]  at (3,-2) {};
\node[fill=black,circle,inner sep=1.5pt]  at (4,-2) {};
\node[fill=black,circle,inner sep=1.5pt]  at (5,-2) {};

\draw (0,-4)--(6,-4);
\draw (1,-4)--(1,-2);
\draw (2,-4)--(2,-2);
\draw (3,-4)--(3,-2);
\draw (4,-4)--(4,-2);
\draw (5,-4)--(5,-2);

\node[fill=black,circle,inner sep=1.5pt]  at (8,-4) {};
\node[fill=black,circle,inner sep=1.5pt]  at (9,-4) {};
\node[fill=black,circle,inner sep=1.5pt]  at (10,-4) {};
\node[fill=black,circle,inner sep=1.5pt]  at (11,-4) {};
\node[fill=black,circle,inner sep=1.5pt]  at (12,-4) {};
\node[fill=black,circle,inner sep=1.5pt]  at (13,-4) {};
\node[fill=black,circle,inner sep=1.5pt]  at (14,-4) {};
\node[fill=black,circle,inner sep=1.5pt]  at (15,-4) {};
\node[fill=black,circle,inner sep=1.5pt]  at (9,-3) {};
\node[fill=black,circle,inner sep=1.5pt]  at (10,-3) {};
\node[fill=black,circle,inner sep=1.5pt]  at (11,-3) {};
\node[fill=black,circle,inner sep=1.5pt]  at (12,-3) {};
\node[fill=black,circle,inner sep=1.5pt]  at (13,-3) {};
\node[fill=black,circle,inner sep=1.5pt]  at (9,-2) {};
\node[fill=black,circle,inner sep=1.5pt]  at (10,-2) {};
\node[fill=black,circle,inner sep=1.5pt]  at (11,-2) {};
\node[fill=black,circle,inner sep=1.5pt]  at (12,-2) {};
\node[fill=black,circle,inner sep=1.5pt]  at (13,-2) {};

\draw (8,-4)--(15,-4);
\draw (9,-4)--(9,-2);
\draw (10,-4)--(10,-2);
\draw (11,-4)--(11,-2);
\draw (12,-4)--(12,-2);
\draw (13,-4)--(13,-2);

\node[fill=black,circle,inner sep=1.5pt]  at (-1,-8) {};
\node[fill=black,circle,inner sep=1.5pt]  at (0,-8) {};
\node[fill=black,circle,inner sep=1.5pt]  at (1,-8) {};
\node[fill=black,circle,inner sep=1.5pt]  at (2,-8) {};
\node[fill=black,circle,inner sep=1.5pt]  at (3,-8) {};
\node[fill=black,circle,inner sep=1.5pt]  at (4,-8) {};
\node[fill=black,circle,inner sep=1.5pt]  at (5,-8) {};
\node[fill=black,circle,inner sep=1.5pt]  at (6,-8) {};
\node[fill=black,circle,inner sep=1.5pt]  at (7,-8) {};
\node[fill=black,circle,inner sep=1.5pt]  at (1,-7) {};
\node[fill=black,circle,inner sep=1.5pt]  at (2,-7) {};
\node[fill=black,circle,inner sep=1.5pt]  at (3,-7) {};
\node[fill=black,circle,inner sep=1.5pt]  at (4,-7) {};
\node[fill=black,circle,inner sep=1.5pt]  at (5,-7) {};
\node[fill=black,circle,inner sep=1.5pt]  at (1,-6) {};
\node[fill=black,circle,inner sep=1.5pt]  at (2,-6) {};
\node[fill=black,circle,inner sep=1.5pt]  at (3,-6) {};
\node[fill=black,circle,inner sep=1.5pt]  at (4,-6) {};
\node[fill=black,circle,inner sep=1.5pt]  at (5,-6) {};

\draw (-1,-8)--(7,-8);
\draw (1,-8)--(1,-6);
\draw (2,-8)--(2,-6);
\draw (3,-8)--(3,-6);
\draw (4,-8)--(4,-6);
\draw (5,-8)--(5,-6);

\node[fill=black,circle,inner sep=1.5pt]  at (8,-8) {};
\node[fill=black,circle,inner sep=1.5pt]  at (9,-8) {};
\node[fill=black,circle,inner sep=1.5pt]  at (10,-8) {};
\node[fill=black,circle,inner sep=1.5pt]  at (11,-8) {};
\node[fill=black,circle,inner sep=1.5pt]  at (12,-8) {};
\node[fill=black,circle,inner sep=1.5pt]  at (13,-8) {};
\node[fill=black,circle,inner sep=1.5pt]  at (14,-8) {};
\node[fill=black,circle,inner sep=1.5pt]  at (15,-8) {};
\node[fill=black,circle,inner sep=1.5pt]  at (9,-7) {};
\node[fill=black,circle,inner sep=1.5pt]  at (10,-7) {};
\node[fill=black,circle,inner sep=1.5pt]  at (11,-7) {};
\node[fill=black,circle,inner sep=1.5pt]  at (12,-7) {};
\node[fill=black,circle,inner sep=1.5pt]  at (13,-7) {};
\node[fill=black,circle,inner sep=1.5pt]  at (14,-7) {};
\node[fill=black,circle,inner sep=1.5pt]  at (9,-6) {};
\node[fill=black,circle,inner sep=1.5pt]  at (10,-6) {};
\node[fill=black,circle,inner sep=1.5pt]  at (11,-6) {};
\node[fill=black,circle,inner sep=1.5pt]  at (12,-6) {};
\node[fill=black,circle,inner sep=1.5pt]  at (13,-6) {};
\node[fill=black,circle,inner sep=1.5pt]  at (14,-6) {};

\draw (8,-8)--(15,-8);
\draw (9,-8)--(9,-6);
\draw (10,-8)--(10,-6);
\draw (11,-8)--(11,-6);
\draw (12,-8)--(12,-6);
\draw (13,-8)--(13,-6);
\draw (14,-8)--(14,-6);

\end{tikzpicture}
\end{center}

Let us therefore consider the sequence of ``comb-shaped'' trees $C_k$ that consist of a ``shaft'' (a path on $k$ vertices) and $k$ ``teeth'' (paths of length $2$ attached to the vertices of the shaft), and let us estimate the number of fully legal matchings in these trees. The edges of a matching $M$ that belong to the shaft divide the tree $C_k$ into pieces. Let us count the number of fully legal possibilities for a piece that contains $\ell$ vertices of the shaft. For the first and last tooth, we only have two possibilities: there can either be no edge of it included in $M$, or precisely one, which is not incident to the shaft. For the other teeth, we have three possibilities: no edge, the edge that is incident to the shaft, or the edge that is not incident to the shaft. Suppose that the first possibility occurs $a$ times, the second $b$ times and the last $c$ times. We must clearly have $a+b+c = \ell-2$. The components of $C_k - M$ that are single edges do not affect legality, so we focus on the component that contains part of the shaft. Its total weight is $3a + b + 6$, $3a + b + 4$, or $3a + b + 2$, depending on whether $0$, $1$ or $2$ edges of the first and last tooth are contained in $M$. On the other hand, the total number of incident matching edges is $b+c+2$, $b+c+3$ or $b+c+4$ respectively. So the legality condition becomes $3a + 4 > c$, $3a + 1 > c$ or $3a - 2 > c$ respectively.

Noting that there are $\binom{\ell-2}{a,b,c}$ ways to distribute the matching edges of $M$ among the different teeth, we eventually find that the number of possibilities for a piece that contains $\ell$ vertices of the shaft is
\[
\sum_{a+b+c = \ell-2} \binom{a+b+c}{a,b,c}\big( [3a + 4 > c] + 2 [3a + 1 > c] + [3a - 2 > c] \big),
\]
using Iverson's notation: $[P] = 1$ if $P$ is true, $[P] = 0$ otherwise.

A fully legal matching of $C_k$ can be obtained by putting together such pieces to a sequence (here we ignore the fact that the first and last piece are slightly different for simplicity, as this does not actually affect the asymptotic growth rate). The generating function for the number of fully legal matchings in $C_k$ is therefore
\begin{align*}
& \frac{1}{1 - \sum_{\ell \geq 2} x^{\ell} \sum_{a+b+c = \ell-2} \binom{\ell-2}{a,b,c}\big( [3a + 4 > c] + 2 [3a + 1 > c] + [3a - 2 > c] \big)} \\
= \\
& \frac{1}{1 - \sum_{a,b,c \geq 0} \binom{a+b+c}{a,b,c}\big( [3a + 4 > c] + 2 [3a + 1 > c] + [3a - 2 > c] \big) x^{a+b+c+2}}
\end{align*}
The asymptotic growth rate of its coefficients is governed by the singularity that is closest to the origin, which is the unique positive real root of the denominator, see Chapter IV in \cite{Flajolet09}. If we let $\rho \approx 0.2633$ be that root, i.e.,
\[
1 = \sum_{a,b,c \geq 0} \binom{a+b+c}{a,b,c}\big( [3a + 4 > c] + 2 [3a + 1 > c] + [3a - 2 > c] \big) \rho^{a+b+c+2} = 3\rho^2 + 8\rho^3 + 24\rho^4 + \cdots,
\]
then the number of fully legal matchings of $C_k$ is $\Theta(\rho^{-k})$. Noting that $C_k$ has $3k$ vertices, we have thus constructed a sequence of trees for which the number of fully legal matchings grows as $\Theta(\beta^n)$, with $\beta = \rho^{-1/3} \approx 1.5603$. This provides us with a lower bound on how far the idea of legal matchings can be used to improve the algorithm further.


\section{Acknowledgements}
We thank Mark Jones for useful discussions. Ruben Meuwese was supported by the Dutch Research Council (NWO) KLEIN 1 grant \emph{Deep
kernelization for phylogenetic discordance}, project number OCENW.KLEIN.305. Stephan Wagner was supported by the Knut and Alice Wallenberg Foundation.




\printbibliography
\Addresses

\newpage

\section*{Appendix A: Fitch's algorithm}

Fitch’s algorithm \cite{Fitch71}, which computes $l_f(T)$ and a corresponding optimal extension $g$ of $f$ to $T$, has two phases. In the first phase, known as the \emph{bottom-up} phase, we start by assigning to each taxon a subset of states consisting of only the state it is assigned by $f$.
The internal nodes of $T$ are assigned subsets of states recursively, as follows. Suppose
a node $p$ has two children u and v, and the bottom-up phase has already assigned
subsets $F(u)$ and $F(v)$ to the two children, respectively. If $F(u) \cap F(v) \neq \emptyset$ then
set $F(p) = F(u) \cap F(v)$ (in which case we say that $p$ is an \emph{intersection} node). If
$F(u) \cap F(v) = \emptyset$ then set $F(p) = F(u) \cup F(v)$ (in which case we say that $p$ is a
\emph{union} node). The number of union nodes in the bottom-up phase is equal to $l_f(T)$.
To actually create an optimal extension $g$ of $f$ to $T$, we require the \emph{top-down} phase of Fitch’s
algorithm. Start at the root $r$ and let $g(r)$ be an arbitrary element in $F(r)$. For an internal
node $u$ with parent $p$, we set $g(u) = g(p)$ (if $g(p) \in F(u)$) and otherwise (i.e.
$g(p) \not \in F(u)$) set $g(u)$ to be an arbitrary element of $F(u)$. Note that the arbitrary choices in the top-down phase can be used to generate multiple distinct optimal extensions.

\newpage
\section*{Appendix B}

The vectors of the set $\mathcal{S}$, with $R = 2793745$:

\begin{align*}
\mathbf{v}_1 &= \alpha^{-2} [1,0,0,0,0]^T, \\
\mathbf{v}_2 &= \alpha^{-6} \Big[2 + \frac{3446}{\sqrt{R}}, 2 + \frac{3170}{\sqrt{R}},0,0,0\Big]^T, \\
\mathbf{v}_3 &= \alpha^{-6} [4,4,0,0,0]^T, \\
\mathbf{v}_4 &= \alpha^{-21} \Big[2016 + \frac{3380832}{\sqrt{R}}, 2208 + \frac{3671904}{\sqrt{R}},0,0,0\Big]^T, \\
\mathbf{v}_5 &= \alpha^{-21} [4032,4416,0,0,0]^T, \\
\mathbf{v}_6 &= \alpha^{-21} \Big[2016 + \frac{3367584}{\sqrt{R}}, 2208 + \frac{3693984}{\sqrt{R}},0,0,0\Big]^T, \\
\mathbf{v}_7 &= \alpha^{-21} \Big[\frac{5628705696}{R} + \frac{3367584}{\sqrt{R}}, \frac{6174294432}{R} +
\frac{3693984}{\sqrt{R}},0,0,0\Big]^T, \\
\mathbf{v}_8 &= \alpha^{-14} \Big[\frac{1969470208}{9R} + \frac{3291861361024}{9R^{3/2}}, \frac{2179611976}{9R} + \frac{3643134552760}{9R^{3/2}},0,0,0\Big]^T, \\
\mathbf{v}_9 &= \alpha^{-10} \Big[12 + \frac{20676}{\sqrt{R}}, 14 + \frac{22466}{\sqrt{R}}, 0, 0, 0\Big]^T, \\
\mathbf{v}_{10} &= \alpha^{-3} \Big[\frac{143768593}{108R} + \frac{85015}{108\sqrt{R}}, \frac{105898423}{72R} + \frac{66145}{72\sqrt{R}},0,0,0\Big]^T, \\
\mathbf{v}_{11} &= \alpha^{-3} \Big[{}- \frac{1625}{108} + \frac{2887105}{108\sqrt{R}} , \frac{1715}{72} - \frac{2737003}{72\sqrt{R}},0,0,0 \Big]^T, \\
\mathbf{v}_{12} &= \alpha^{-3} \Big[ \frac{108745}{108} - \frac{65 \sqrt{R}}{108}, \frac{28441}{24} - \frac{17 \sqrt{R}}{24},0,0,0 \Big]^T, \\
\mathbf{v}_{13} &= \alpha^{-3} \Big[ \frac12 + \frac{1447}{2 \sqrt{R}}, \frac12 + \frac{1999}{2\sqrt{R}},0,0,0 \Big]^T, \\
\mathbf{v}_{14} &= \alpha^{-14} \Big[ \frac{201818664}{R} + \frac{120744}{\sqrt{R}}, \frac{267240144}{R}+ \frac{159888}{\sqrt{R}},0,0,0 \Big]^T, \\
\mathbf{v}_{15} &= \alpha^{-14} \Big[ 72 + \frac{120744}{\sqrt{R}}, 96 + \frac{159888}{\sqrt{R}}, 0, 0, 0 \Big]^T, \\
\mathbf{v}_{16} &= \alpha^{-14} [144, 192, 0, 0, 0]^T, \\
\mathbf{v}_{17} &= \alpha^{-7} \Big[ \frac{31}{18} + \frac{116617}{18\sqrt{R}}, \frac{281}{54} + \frac{208991}{54\sqrt{R}}, 0, 0, 0\Big]^T, \\
\mathbf{v}_{18} &= \alpha^{-7} \Big[ \frac{108745}{18} - \frac{65\sqrt{R}}{18}, \frac{219163}{27} - \frac{131 \sqrt{R}}{27}, 0, 0, 0\Big]^T, \\
\mathbf{v}_{19} &= \Big[ \frac{8142156817}{23328R} + \frac{3615127}{23328 \sqrt{R}},
\frac{2689931479}{11664R} + \frac{4131265}{11664\sqrt{R}}, 0, 0, 0\Big]^T, \\
\mathbf{v}_{20} &= \Big[ -\frac{2104505}{23328} + \frac{3526059745}{23328\sqrt{R}}, \frac{2807237}{23328} - \frac{4680672973}{23328 \sqrt{R}}, 0, 0, 0\Big]^T, \\
\mathbf{v}_{21} &= \Big[ \frac{11814523825}{23328} - \frac{7068425 \sqrt{R}}{23328}, \frac{1999380955}{2916} - \frac{1196195 \sqrt{R}}{2916}, 0, 0, 0\Big]^T, \\
\mathbf{v}_{22} &= \Big[ \frac{7345}{432} - \frac{12119705}{432\sqrt{R}}, - \frac{6197}{288} + \frac{10499773}{288\sqrt{R}}, 0, 0, 0\Big]^T, \\
\mathbf{v}_{23} &= \Big[ \frac{2443777}{8R} + \frac{1447}{8\sqrt{R}}, \frac{3242521}{8R} + \frac{1999}{8\sqrt{R}}, 0, 0, 0\Big]^T, 
\end{align*}

\newpage

\begin{align*}
\mathbf{v}_{24} &= \Big[ \frac{32805161}{108R} + \frac{54881040239}{108R^{3/2}}, \frac{29641217}{72R} + \frac{49498725911}{72R^{3/2}}, 0, 0, 0\Big]^T, \\
\mathbf{v}_{25} &= \Big[ \frac{1237818669070513}{1458R^2} + \frac{740509100311}{1458R^{3/2}}, \frac{838121265457801}{729R^2} + \frac{501401051935}{729R^{3/2}}, 0, 0, 0\Big]^T, \\ 
\mathbf{v}_{26} &= \alpha^{-15} \Big[ \frac{71754261114955960}{81R^2} + \frac{119933657317951854472}{81 R^{5/2}}, \\
&\qquad \frac{291727305240092752}{243R^2} + \frac{487607592266798252080}{243R^{5/2}}, 0, 0, 0\Big]^T, \\
\mathbf{v}_{27} &= \alpha^{-8} \Big[ \frac{140725263328052732114393}{1458R^3} + \frac{84193524123331244303}{1458R^{5/2}}, \\
&\qquad \frac{286283955528410226595427}{2187R^3} + \frac{171278806193160289301}{2187R^{5/2}}, 0, 0, 0\Big]^T, \\
\mathbf{v}_{28} &= \alpha^{-4} \Big[ \frac{10512907255001141}{1944R^2} + \frac{6289683368723}{1944R^{3/2}}, \frac{10709134162880129}{1458R^2} + \frac{6407069939495}{1458R^{3/2}}, 0, 0, 0\Big]^T, \\
\mathbf{v}_{29} &= \Big[ \frac{794956577}{2592R} + \frac{464135}{2592\sqrt{R}}, \frac{524137043}{1296R} + \frac{325541}{1296\sqrt{R}}, 0, 0, 0\Big]^T, \\
\mathbf{v}_{30} &= \alpha^{-15} \Big[ \frac{17072730067}{54R} + \frac{28535661518197}{54R^{3/2}}, \frac{3873220609}{9R} + \frac{6473856199063}{9 R^{3/2}}, 0, 0, 0\Big]^T, \\
\mathbf{v}_{31} &= \alpha^{-8} \Big[ \frac{803579524335268753}{23328R^2} + \frac{480751987830295}{23328R^{3/2}}, \\
&\qquad \frac{547329217239002015}{11664R^2} + \frac{327434508313145}{11664R^{3/2}}, 0, 0, 0\Big]^T, \\
\mathbf{v}_{32} &= \alpha^{-8} \Big[ \frac{293738136445}{23328R} + \frac{470526609087163}{23328R^{3/2}}, \frac{197305889945}{11664R} + \frac{325111894094399}{11664R^{3/2}}, 0, 0, 0\Big]^T, \\
\mathbf{v}_{33} &= \alpha^{-8} \Big[ \frac{8133848923273}{23328R} - \frac{4522191905}{23328\sqrt{R}}, -\frac{7653084333757}{11664R} + \frac{4813128245}{11664\sqrt{R}}, 0, 0, 0\Big]^T, \\
\mathbf{v}_{34} &= \alpha^{-2} \Big[ \frac{34964309}{54R} + \frac{21179}{54\sqrt{R}}, \frac{29330411}{27R} + \frac{17753}{27\sqrt{R}}, 0, 0, 0\Big]^T, \\
\mathbf{v}_{35} &= \alpha^{-17} \Big[ \frac{2039517464}{3R} + \frac{3408952424936}{3R^{3/2}}, \frac{10267634576}{9R} + \frac{17161853188592}{9 R^{3/2}}, 0, 0, 0\Big]^T, \\
\mathbf{v}_{36} &= \alpha^{-13} \Big[ 38 + \frac{63818}{\sqrt{R}}, 64 + \frac{106960}{\sqrt{R}}, 0, 0, 0\Big]^T, \\
\mathbf{v}_{37} &= \alpha^{-6} \Big[ \frac{892769641}{216R} + \frac{538495}{216\sqrt{R}}, \frac{62416754}{9R} + \frac{37790}{9\sqrt{R}}, 0, 0, 0\Big]^T, \\
\mathbf{v}_{38} &= \alpha^{-6} \Big[{}-\frac{1355}{216} + \frac{3337411}{216\sqrt{R}}, \frac{5}{2} + \frac{8339}{2\sqrt{R}}, 0, 0, 0\Big]^T, \\
\mathbf{v}_{39} &= \alpha^{-4} [1,2,0,0,0]^T, \\
\mathbf{v}_{40} &= \alpha^{-4} \Big[\frac{1895}{216} - \frac{2436799}{216\sqrt{R}}, 0, 0, - \frac{1625}{108} + \frac{2887105}{108\sqrt{R}}, 0\Big]^T, \\
\mathbf{v}_{41} &= \alpha^{-4} \Big[ \frac{605232455}{216R} + \frac{368465}{216\sqrt{R}}, 0, 0, \frac{143768593}{108R} + \frac{85015}{108\sqrt{R}}, 0\Big]^T, \\
\mathbf{v}_{42} &= \alpha^{-11} \Big[ 26 + \frac{43142}{\sqrt{R}}, 0, 0, 12 + \frac{20676}{\sqrt{R}}, 0 \Big]^T, \\
\mathbf{v}_{43} &= \alpha^{-15} \Big[ \frac{4149082184}{9R} + \frac{6934995913784}{9R^{3/2}}, 0, 0, \frac{1969470208}{9R} + \frac{3291861361024}{9R^{3/2}}, 0 \Big]^T, \\
\mathbf{v}_{44} &= \Big[ \frac{23696513}{54R} + \frac{14327}{54\sqrt{R}}, 0, 0, 
\frac{1877966}{9R} + \frac{1142}{9\sqrt{R}}, 0 \Big]^T, \\
\mathbf{v}_{45} &= \Big[ \frac18 + \frac{2551}{8\sqrt{R}}, 0, 0, \frac18 + \frac{343}{8\sqrt{R}}, 0 \Big]^T, 
\end{align*}

\newpage

\begin{align*}
\mathbf{v}_{46} &= \Big[ \frac{18403}{54} - \frac{11\sqrt{R}}{54}, 0, 0, \frac{11711}{72} - \frac{7\sqrt{R}}{72}, 0 \Big]^T, \\
\mathbf{v}_{47} &= \Big[ \frac{73}{216} - \frac{8081}{216\sqrt{R}}, 0, 0, - \frac{7}{36} + \frac{20783}{36\sqrt{R}}, 0 \Big]^T, \\
\mathbf{v}_{48} &= \alpha^{-7} [8,0,0,4,0]^T, \\
\mathbf{v}_{49} &= \alpha^{-7} \Big[ 4 + \frac{6616}{\sqrt{R}}, 0, 0, 2 + \frac{3446}{\sqrt{R}}, 0 \Big]^T, \\
\mathbf{v}_{50} &= \alpha^{-3} [1, 0, 0, 1, 0]^T, \\
\mathbf{v}_{51} &= \alpha^{-5} \Big[ \frac{1895}{216} - \frac{2436799}{216\sqrt{R}}, 0, -\frac{1625}{108} + \frac{2887105}{108\sqrt{R}}, \frac{1895}{216} - \frac{2436799}{216\sqrt{R}},
  0\Big]^T, \\
\mathbf{v}_{52} &= \alpha^{-5} \Big[ \frac{605232455}{216R} + \frac{368465}{216\sqrt{R}}, 0, 
 \frac{143768593}{108R} + \frac{85015}{108\sqrt{R}}, \frac{605232455}{216R} + \frac{368465}{216\sqrt{R}}, 0\Big]^T, \\
\mathbf{v}_{53} &= \alpha^{-12} \Big[26 + \frac{43142}{\sqrt{R}}, 0, 12 + \frac{20676}{\sqrt{R}}, 
 26 + \frac{43142}{\sqrt{R}}, 0 \Big]^T, \\
\mathbf{v}_{54} &= \alpha^{-16} \Big[ \frac{4149082184}{9R} + \frac{6934995913784}{9R^{3/2}}, 0, 
 \frac{1969470208}{9R} + \frac{3291861361024}{9R^{3/2}}, \\
&\qquad \frac{4149082184}{9R} + \frac{6934995913784}{9R^{3/2}}, 0 \Big]^T, \\
\mathbf{v}_{55} &= \alpha^{-1} \Big[ \frac{23696513}{54R} + \frac{14327}{54\sqrt{R}}, 0, \frac{1877966}{9R} + \frac{1142}{9\sqrt{R}}, \frac{23696513}{54R} + \frac{14327}{54\sqrt{R}}, 0 \Big]^T, \\
\mathbf{v}_{56} &= \alpha^{-1} \Big[ \frac18 + \frac{2551}{8 \sqrt{R}}, 0, \frac18 + \frac{343}{8\sqrt{R}}, 
 \frac18 + \frac{2551}{8\sqrt{R}}, 0 \Big]^T, \\
\mathbf{v}_{57} &= \alpha^{-1} \Big[ \frac{18403}{54} - \frac{11 \sqrt{R}}{54}, 0, \frac{11711}{72} - \frac{7\sqrt{R}}{72}, \frac{18403}{54} - \frac{11\sqrt{R}}{54}, 0 \Big]^T \\
\mathbf{v}_{58} &= \alpha^{-1} \Big[ \frac{73}{216} - \frac{8081}{216\sqrt{R}}, 0, - \frac{7}{36} + \frac{20783}{36 \sqrt{R}}, \frac{73}{216} - \frac{8081}{216\sqrt{R}}, 0\Big]^T, \\
\mathbf{v}_{59} &= \alpha^{-8} [8, 0, 4, 8, 0]^T, \\
\mathbf{v}_{60} &= \alpha^{-8} \Big[4 + \frac{6616}{\sqrt{R}}, 0, 2 + \frac{3446}{\sqrt{R}}, 4 + \frac{6616}{\sqrt{R}}, 0 \Big]^T, \\
\mathbf{v}_{61} &= \alpha^{-4} [1,0,1,1,0]^T, \\
\mathbf{v}_{62} &= \alpha^{-1} [0, 0, 0, 0, 1]^T. \\
\end{align*}

The set $S$, consisting of vectors $[v,w,x,y,z]^T$, can be split into 4 groups:
\begin{enumerate}
    \item Vectors 1 to 39 have $x=y=z=0$ and the rest is positive.
    \item Vectors 40 to 50 have $w=x=z=0$ and the rest is positive.
    \item Vectors 51 to 61 have $w=z=0$ and $v=y$ and the rest is positive.
    \item Vector 62 has $v=w=x=y=0$ and $z>0$.
\end{enumerate}

Now, if we look at $B(a,b)$ we can observe that we can split the image into 4 groups. When we say ``combine" two vectors $a$ and $b$ we mean look at the result of $B(a,b)$.

\begin{enumerate}[label=\roman*]
    \item Vector combinations from every group except vertex 62 will result in a vector that has $x=y=z=0$. This is because $z=0$ in all the vectors we combine.
    \item Combining vector 62 with group 1 above will result in a vector that has $w=x=z=0$.
    \item Combining vector 62 with group 2 above will result in a vector that has $w=z=0$ and $v=y$ and the rest is positive.
    \item Combining vector 62 with group 3 above will result in a vector that has $w=z=0$ and $x=y$ and the rest is positive.
\end{enumerate}

These observations will make checking that $S$ has the desired property much easier; we provide further details in the supplementary Mathematica files:
\url{https://github.com/skelk2001/legal-matchings}

\section*{Appendix C}

Let $M_n$ denote the maximum number of legal matchings in a tree with $n$ nodes. In the following, we determine a lower bound on the rate of growth of $M_n$. Recall that a \emph{good} tree is a tree with maximum degree 3, where no two degree 2 nodes are adjacent.

\begin{Lemma}
\label{lem:lb}
$M_n = \Omega(\alpha^n)$ with $\alpha = (13384+8\sqrt{2793745})^{1/22} \approx 1.5895$. This lower bound is achieved on good trees.
\end{Lemma}

\begin{proof}
We provide an explicit construction. To this end, let us introduce some notation. For a rooted tree $T$ with root $r$, let $z(T)$ be the total number of matchings of $T$ and $z_0(T)$ the number of matchings  where the root is not covered (i.e., the root is not an end of one of the matching edges).

Now let us construct a sequence of rooted trees $T_k$ with a repeating pattern of length $22$. The starting point can be any good rooted tree $T_0$; then all further trees in the sequence are good trees as well.

\begin{figure}[!h]
\begin{center}
\begin{tikzpicture}[scale=0.7]
\node[fill=black,circle,inner sep=1.5pt]  at (0,0) {};
\node[fill=black,circle,inner sep=1.5pt]  at (1,0) {};
\node[fill=black,circle,inner sep=1.5pt]  at (2,0) {};
\node[fill=black,circle,inner sep=1.5pt]  at (3,0) {};
\node[fill=black,circle,inner sep=1.5pt]  at (4,0) {};
\node[fill=black,circle,inner sep=1.5pt]  at (5,0) {};
\node[fill=black,circle,inner sep=1.5pt]  at (6,0) {};
\node[fill=black,circle,inner sep=1.5pt]  at (7,0) {};
\node[fill=black,circle,inner sep=1.5pt]  at (8,0) {};
\node[fill=black,circle,inner sep=1.5pt]  at (9,0) {};
\node[fill=black,circle,inner sep=1.5pt]  at (1,1) {};
\node[fill=black,circle,inner sep=1.5pt]  at (1,2) {};
\node[fill=black,circle,inner sep=1.5pt]  at (2,-1) {};
\node[fill=black,circle,inner sep=1.5pt]  at (2,-2) {};
\node[fill=black,circle,inner sep=1.5pt]  at (4,1) {};
\node[fill=black,circle,inner sep=1.5pt]  at (4,2) {};
\node[fill=black,circle,inner sep=1.5pt]  at (3,3) {};
\node[fill=black,circle,inner sep=1.5pt]  at (5,3) {};
\node[fill=black,circle,inner sep=1.5pt]  at (3,4) {};
\node[fill=black,circle,inner sep=1.5pt]  at (5,4) {};
\node[fill=black,circle,inner sep=1.5pt]  at (5,-1) {};
\node[fill=black,circle,inner sep=1.5pt]  at (5,-2) {};
\node[fill=black,circle,inner sep=1.5pt]  at (4,-3) {};
\node[fill=black,circle,inner sep=1.5pt]  at (6,-3) {};
\node[fill=black,circle,inner sep=1.5pt]  at (4,-4) {};
\node[fill=black,circle,inner sep=1.5pt]  at (6,-4) {};
\node[fill=black,circle,inner sep=1.5pt]  at (7,1) {};
\node[fill=black,circle,inner sep=1.5pt]  at (7,2) {};
\node[fill=black,circle,inner sep=1.5pt]  at (8,-1) {};
\node[fill=black,circle,inner sep=1.5pt]  at (8,-2) {};

\draw (0,0)--(9,0);
\draw (1,0)--(1,2);
\draw (2,0)--(2,-2);
\draw (4,0)--(4,2);
\draw (3,4)--(3,3)--(4,2)--(5,3)--(5,4);
\draw (5,0)--(5,-2);
\draw (4,-4)--(4,-3)--(5,-2)--(6,-3)--(6,-4);
\draw (7,0)--(7,2);
\draw (8,0)--(8,-2);

\node at (-1,0) {$\cdots$};
\node at (9,-0.5) {$r_k$};
\node at (3,-0.5) {$r_{k+1}$};
\node at (10.5,0) {$T_k$};
\node at (3,-4) {$T_{k+1}$};

\draw [dashed] (8.5,0)--(12.5,3)--(12.5,-3)--(8.5,0);
\draw [dashed] (2.5,-4.5)--(2.5,4.5)--(13,4.5)--(13,-4.5)--(2.5,-4.5);

\end{tikzpicture}
\end{center}
\caption{Construction of a sequence of trees with many legal matchings.}\label{fig:max_construct}
\end{figure}
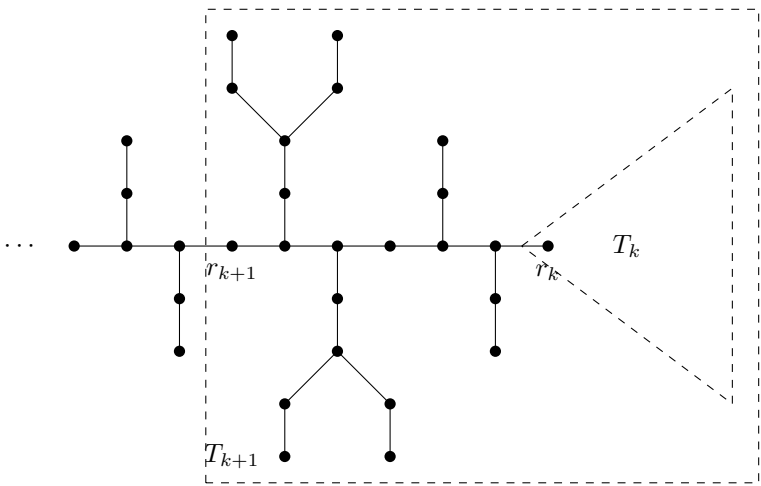

The tree $T_{k+1}$ is constructed by attaching a piece of $22$ nodes (and $22$ edges) to the root of $T_k$ and assigning a new root $r_{k+1}$ as indicated in Figure~\ref{fig:max_construct}. Now we can determine a recursion for $z(T_{k+1})$ and $z_0(T_{k+1})$ in terms of $z(T_k)$ and $z_0(T_k)$. There are $10144$ matchings of the $22$ new edges that cover $r_k$, and $19888$ matchings that do not cover it. The former can be extended with any matching of $T_k$ that does not cover $r_k$, the latter with an arbitrary matching of $T_k$. Hence we have
$$z(T_{k+1}) = 19888 z(T_k)  + 10144 z_0(T_k).$$
Likewise, considering only matchings that leave $r_{k+1}$ uncovered, we obtain
$$z_0(T_{k+1}) = 13456 z(T_k)  + 6880 z_0(T_k).$$
Hence we have
$$\begin{bmatrix} z(T_k) \\ z_0(T_k) \end{bmatrix} = \begin{bmatrix} 19888 & 10144 \\ 13456 & 6880 \end{bmatrix}^k \begin{bmatrix} z(T_0) \\ z_0(T_0) \end{bmatrix}.$$
Since the eigenvalues of the $2 \times 2$-matrix are $13384 \pm 8\sqrt{2793745}$, it follows that $z(T_k) = \Theta\big((13384 + 8\sqrt{2793745})^k\big)$. Note that $T_k$ has $n = 22k + O(1)$ nodes, and that we can cover all possible congruence classes modulo $22$ by choosing $T_0$ accordingly. Hence we have $M_n = \Omega(\alpha^n)$ with $\alpha = (13384+8\sqrt{2793745})^{1/22} \approx 1.5895$.
\end{proof}





\newpage

\end{document}